\documentclass[journal,twoside,web]{ieeecolor}
\usepackage{generic}
\usepackage{cite}
\usepackage{amsmath,amssymb,amsfonts}
\usepackage{algorithmic}
\usepackage{graphicx}
\usepackage{algorithm,algorithmic}
\usepackage{hyperref}
\hypersetup{hidelinks=true}
\usepackage{multirow}
\usepackage{textcomp}
\def\BibTeX{{\rm B\kern-.05em{\sc i\kern-.025em b}\kern-.08em
    T\kern-.1667em\lower.7ex\hbox{E}\kern-.125emX}}
\usepackage{soul,color}

\newtheorem{defi}{Definition}
\newtheorem{thm}{Theorem}
\newtheorem{prop}{Proposition}
\newtheorem{corol}{Corollary}

\newtheorem{lem}{Lemma}
\newtheorem{Ass}{Assumption}
\usepackage{booktabs}       
\usepackage{tabularx}      
\usepackage{makecell} 
\newcommand{\ubar}[1]{\underline{#1}}
\begin{document}
\title{Concave Comparison Functions for Accelerating Constrained Lyapunov Decay}

\author{Shuyuan Fan,  Guanru Pan, and Herbert Werner
\thanks{
The authors are with the Institute of Control Systems,  Hamburg University of Technology,
Harburger Schloßstraße 22a,
21079 Hamburg, Germany (e-mail: shuyuan.fan@tuhh.de; 
       guanru.pan@tuhh.de; h.werner@tuhh.de )  }
}

\maketitle

\begin{abstract}
What limits how fast a Lyapunov function can decay under input bounds?
We address this question by showing how the \emph{shape} of Lyapunov comparison functions governs guaranteed decay for control–affine systems. Using a windowed \emph{nominal exponential rate}  together with the endpoint cap induced by actuator limits, we establish a strict ordering: concave comparison functions strictly outperform linear and convex ones, and \emph{strict concavity is necessary} to improve the best achievable global exponential rate under a fixed endpoint cap. We derive a computable lower bound on the required actuation level for a target nominal rate and show that only concave shaping can reduce this level under the endpoint cap. We then establish a feasibility-preserving acceleration result: whenever a margin exists on a sublevel set, a feasible linear comparison can be replaced by a concave one that preserves feasibility while strictly increasing the guaranteed windowed decay. Finally, we give a tunable rational concave factor with controlled slope that yields a constructive design and integrates  with CLF–QP, as illustrated by examples.
\end{abstract}
\begin{IEEEkeywords}
Lyapunov methods, concave comparison functions, input saturation, CLF–QP
\end{IEEEkeywords}

\section{Introduction}
Safety–critical systems (e.g., robotics, automotive, power, and aerospace systems) routinely operate under hard actuator bounds. These bounds cap the certified Lyapunov decay rate and thus the guaranteed settling time. This paper asks: \emph{What fundamentally limits Lyapunov decay under input bounds, and how should the comparison function be shaped to exploit available actuation most effectively?}

\subsection{Comparison principle}
The comparison principle is a cornerstone of Lyapunov analysis for nonlinear systems. In its basic form (e.g., \cite[Lem.~4.4]{lin1996smooth}, \cite[Secs.~3.4,\,4.4]{khalil2002nonlinear}, and \cite[App.~A]{Mironchenkbook}), if a smooth positive–definite $V:\mathbb{R}^n\to\mathbb{R}_{\geq0}$ satisfies
\begin{equation}\label{eq:intro-V-ineq}
  \dot V(x) \le -\alpha\big(V(x)\big),
\end{equation}
for some class-$\mathcal K$ function $\alpha:\mathbb{R}_{\geq0}\to\mathbb{R}_{\geq0}$ (cf. the \emph{comparison function} or \emph{Lyapunov decay rate}), then we obtain the \emph{comparison scalar ODE} (upper-bound system):
\begin{equation}\label{eq:intro-ub}
  \dot y = -\alpha(y),\qquad y(0)=V(x_0),
\end{equation}
which provides a class-$\mathcal{KL}$ estimate that
\begin{align}\label{eq:intro KL}
    \dot{V}(x(t))\leq \beta(y_0,t), \quad \forall t\geq0.
\end{align}
where $\beta(y_0,t)$ is the solution of \eqref{eq:intro-ub}.
 Thus, the upper–bound system \eqref{eq:intro-ub} yields a rigorous worst–case decay bound on $V(x(t))$ for estimation and analysis.

\subsection{Background and Motivation}
Beyond analysis, comparison functions play a constructive role in the control Lyapunov function (CLF) based control. In the CLF framework,  originating with Artstein \cite{ArtsteinCLF} and made explicit by Sontag’s universal formula \cite{sontag1989universal}, the decay constraint  $\dot{V}(x,u)\leq-\alpha(V)$ (cf. inequality \eqref{eq:V-ineq}) is enforced as a design condition to certify stabilizability and parameterize the admissible feedback set.
 In practice, three types of comparison functions are mainly considered:
(i) the \emph{linear} comparison $\alpha(V)=\sigma V$ yielding exponential decay with rate $\sigma$,
(ii) the \emph{root}  comparison $\alpha(V)=C V^{p}$  with $C>0$ and $p\in(0,1)$ guaranteeing finite-time convergence (see 
\cite{FTC}), and
(iii) \emph{polynomial} mixtures such as $\alpha(V)=C_1 V^{p}+C_2 V^{q}$ with $0<p<1<q$,
 invoked for fixed-time convergence (see \cite{Fixtime}).

The latter two are non-Lipschitz at the origin and are rarely used in optimization-based control
because they complicate KKT regularity and state–control mappings\cite{mestres2025regularity}. In addition, the over-aggressive decay around the origin, induced by the non-Lipschitz properties, contradicts the well-posedness theorem (cf. \cite[Thm 3.1]{khalil2002nonlinear}) and causes chattering (see \cite{ChatteringAnalysis}). Consequently, the linear
type $\alpha(V)=\sigma V$ has become the default performance constraint in many CLF-driven
designs, e.g., CLF/CBF-QP (see \cite{CLFreview,inverseCLF,CLF-bipedal,ames2019control}), control contraction metrics based optimal control (see \cite{CCM}), and LMI control formulations (see \cite{LMI}), with the exponential rate $\sigma$ serving as the performance metric.

In a constrained CLF setting \cite{Mhaskar2005,constraintCLF,Mhaskar2006}, actuator limits impose level-wise caps on admissible Lyapunov decay, thereby constraining feasible CLFs and their decay rates. From the CLF--QP perspective with a linear comparison $\alpha(V)=\sigma V$, feasibility under input limits is typically maintained by either (i) reducing $\sigma$ (conservative), or (ii) introducing a penalized slack $\delta\ge0$ in the QP, which relaxes the constraint and thus weakens the certified decay. With the linear comparison, $\sigma$ becomes the sole performance objective; acceleration strategies therefore either \emph{schedule} it (e.g., rapidly exponentially stabilizing CLF--QP \cite{ames2014rapidly}) or \emph{optimize} it online as a decision variable (e.g., flexible ES--CLF--QP \cite{FlexCLF} and optimal-decay CLF--QP \cite{OD-CLF-QP}). This single-rate reliance limits design flexibility for decay shaping under actuator saturation.

Rather than only scaling the rate (increasing $\sigma$), we ask a complementary question: How does the \emph{shape} of $\alpha$ (linear, convex, concave, or mixtures) mediate the trade–off between feasibility and guaranteed decay under fixed input bounds? Figure~\ref{fig:sketch_of_comparison} gives a quick visual: under a common endpoint cap $\alpha(c)$, a strictly concave $\alpha$ lies above the linear and convex ones on $(0,c)$ and thus certifies a larger guaranteed decay (equivalently, a larger windowed nominal rate $\sigma_{\alpha}(\epsilon,c)$).

\begin{figure}
    \centering
    \includegraphics[width=0.7\linewidth]{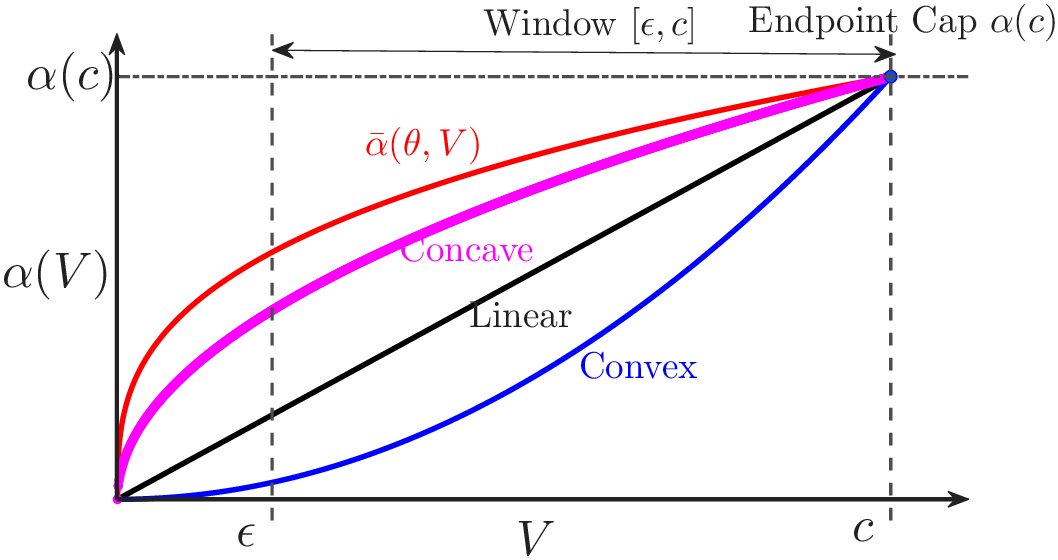}
    \caption{Sketch: decay cap, endpoint cap, evaluation window $[\epsilon,c]$. The concave curve lies above linear and convex comparison, hence the largest $\sigma_\alpha(\epsilon,c)$ is guaranteed.}
    \label{fig:sketch_of_comparison}
\end{figure}

\subsection{Contributions}
We develop a shape-aware framework for Lyapunov decay under input saturation, evaluated through windowed performance metrics. Our main contributions are:
\begin{itemize}
  \item \textbf{Decay acceleration by the concave comparison.}
  For the upper-bound ODE \eqref{eq:intro-ub}, on any window $[\epsilon,c]$ and under a common endpoint cap, strictly concave comparisons guarantee a strictly larger windowed nominal rate $\sigma_\alpha(\epsilon,c)$ than linear or convex ones. Moreover, any improvement over the linear baseline with the same endpoint cap requires \emph{strict} concavity on a nontrivial subset of $(0,c)$.
  
  \item \textbf{Lower peak actuation for higher certified rate.}
  We derive the actuation level required to achieve a target windowed rate and show that, under a linear comparison, higher nominal rates demand higher (endpoint–dominated) actuation level. In contrast, concave shaping achieves a strictly larger $\sigma_\alpha(\epsilon,c)$ with a smaller actuation level.

  \item \textbf{Feasibility-preserving acceleration.}
  Whenever a margin exists on a sublevel set, a feasible linear comparison can be replaced by a concave one that preserves feasibility while increasing the guaranteed decay over the window.
  
  \item \textbf{Constructive concave design.}
  We introduce a Lipschitz rational concave factor with closed-form normalization and slope control, enabling practical tuning and  integration with CLF–QP.
\end{itemize}

\subsection*{Notation}
$\mathbb{R}$, $\mathbb{R}_{\ge0}$, and $\mathbb{R}_{>0}$ denote the real, nonnegative real, and positive real numbers.
For $x\in\mathbb{R}^n$, $\|x\|$ is the Euclidean norm,  $\|x\|_1=\sum_{i=1}^{n}|x_i|$,
and $\|x\|_\infty=\max_i |x_i|$. 
For a matrix $A$, $A^\top$ is its transpose;  $\mathrm{diag}(\cdot)$ builds a diagonal matrix; $A\succ 0$ ($A\prec0$) means positive (negative) definite. For $A\succ 0$, $\lambda_{\min}(A)$ and $\lambda_{\max}(A)$ denote its smallest and largest eigenvalues. 
All functions are locally Lipschitz unless stated otherwise; $C^k$ denotes $k$-times continuously differentiable. For $f:\mathbb{R}\to\mathbb{R}$, $f'(x)=\frac{\mathrm{d}f}{dx}$. 
 For $f:\mathbb{R}^n\to\mathbb{R}$, the gradient is
$\nabla f(x)=\big[\partial f/\partial x_1,\dots,\partial f/\partial x_n\big]^\top$.
Time derivatives are denoted by a dot, e.g., $\dot x$. Lie derivatives $L_f V:=\nabla V^\top f$, $L_g V:=\nabla V^\top g$. $\alpha\in\mathcal K$ means $\alpha:[0,a)\to[0,\infty)$ is continuous, $\alpha(0)=0$, and strictly increasing on $[0,a)$. $\alpha\in\mathcal{K}_{vex}/\mathcal{K}_{cave}/\mathcal{K}_{l}$ implies $\alpha\in\mathcal{K}$ and $\alpha$ is strictly convex/strictly concave/linear, respectively.

\section{Preliminaries}\label{sec:prelim}
We consider the control‐affine system
\begin{equation}\label{eq:system}
  \dot x = f(x) + g(x)u,
\end{equation}
with $x\in\mathbb{X}\subset\mathbb{R}^n$ and $u\in\mathbb{U}(\theta)\subset\mathbb{R}^m$, where $f:\mathbb{X}\to\mathbb{R}^n$ and $g:\mathbb{X}\to\mathbb{R}^{n\times m}$ are locally Lipschitz with $f(0)=0$  and with $g(0)\ne0$. Inputs are bounded in:
\begin{align}
  \mathbb{U}(\theta) \doteq \{\,u\in\mathbb{R}^m \mid \|u\|_\infty \le \theta\,\}
\end{align}
where $\theta>0$ represents  the \emph{actuation level}. 
For an initial condition $x(0)=x_0$ and an admissible input $u(\cdot)\in \mathbb{U}(\theta)$ (measurable, essentially bounded), let $\phi(t,x_0,u)$ denote the Carathéodory solution of \eqref{eq:system}. We assume $\mathbb{X}$ is compact, contains the origin, and is \emph{control–invariant} with respect to $\mathbb{U}(\theta)$; i.e., for every $x_0\in\mathbb{X}$ there exists an input $u(\cdot)\in\mathbb{U}(\theta)$ such that $\phi(t,x_0,u)\in\mathbb{X}$ for all $t\ge0$.

\subsection{Control Lyapunov functions and comparison}\label{subsec:clf}
We begin with a standard CLF definition.

\begin{defi}[Control Lyapunov function]
A smooth, positive–definite $V:\mathbb{R}^n\to\mathbb{R}_{\ge0}$ is a CLF for \eqref{eq:system} on $\mathbb{X}$ with inputs in $\mathbb{U}(\theta)$ if there exist
$0<k_1\le k_2$ and a class-$\mathcal{K}$ function $\alpha$ such that, for all $x\in\mathbb{X}$,
\begin{subequations}\label{eq:clf}
\begin{equation}
  k_1\|x\|^2 \;\le\; V(x) \;\le\; k_2\|x\|^2,
\label{eq:clf-bounds}
\end{equation}
\begin{equation}
  \inf_{u\in\mathbb{U}(\theta)}\big\{\, L_fV(x)+L_gV(x)\,u+\alpha\big(V(x)\big) \big\}\le0.
\label{eq:clf-decay}
\end{equation}
\end{subequations}
\end{defi}
The corresponding admissible feedback set (same as \cite{ames2014rapidly})
\[
  K(x)\doteq\Big\{u\in\mathbb{U}(\theta): L_fV(x)+L_gV(x)u\le-\alpha\big(V(x)\big)\Big\}
\]
is nonempty for all $x\in\mathbb{X}$ by \eqref{eq:clf-decay}. Any (measurable) selection $u(x)\in K(x)$ enforces the differential inequality
\begin{equation}\label{eq:V-ineq}
  \dot V(x,u)\doteq L_fV(x)+L_gV(x)u\le-\alpha\big(V(x)\big),
\end{equation}
which aligns with \eqref{eq:intro-V-ineq} and yields the scalar comparison bound $V(x(t))\le y(t)$, where $y(t)$ solves \eqref{eq:intro-ub}.
\subsubsection*{Linear comparison (exponential stabilizing CLF)}
If \eqref{eq:clf} holds with the linear comparison $\alpha(V)=\sigma V$ ($\sigma>0$), then for any $u(x)\in K(x)$ the closed-loop system
\begin{align}
     \dot{x}=f(x)+g(x)u(x)
\end{align}
has the solution $x(t)$ satisfying the explicit estimate
\begin{equation}\label{eq:exp-bound}
  \|x(t)\|\;\le\;\sqrt{\frac{k_2}{k_1}}e^{-\frac{\sigma}{2}t}\|x(0)\|,\qquad \forall\,t\ge0,
\end{equation}
i.e., global exponential convergence with rate $\sigma/2$. Thus, larger $\sigma$ yields faster state convergence, linking the Lyapunov decay enforced by \eqref{eq:V-ineq} to state bounds via the comparison framework \eqref{eq:intro-V-ineq}-\eqref{eq:intro KL}.

\subsection{Windowed decay metrics}\label{subsec:metrics}
Consider $0<\epsilon<c\le c_{\max}$ and  define the sublevel set
\[
  \Omega(V,c)\doteq\{x : V(x)\le c\},\quad
  c_{\max}\doteq\sup\{v:\Omega(V,v)\subseteq \mathbb{X}\}.
\]
Given an evaluation window $[\epsilon,c]$, we introduce windowed decay metrics for evaluating performance on this window.

\begin{defi}[Nominal rate]
Let the \emph{crossing time} $T(\epsilon,c)$ be the first time at which the trajectory $V(\phi(t,x_0,u))$, initialized on the level set $V(x_0)=c$, enters $\Omega(V,\epsilon)$. The associated \emph{nominal exponential decay rate (simply nominal rate)} over the window is
\begin{equation}\label{eq:sigma-nom}
  \sigma_{\mathrm{nom}}(\epsilon,c)\doteq\frac{\ln(c/\epsilon)}{T(\epsilon,c)}.
\end{equation}
\end{defi}

For the comparison system \eqref{eq:intro-ub} with  $\alpha\in\mathcal{K}$, let  $T_{\alpha}(\epsilon,c)$ denote the crossing time of its solution $y(t)$ with $y(0)=c$. By  the chain rule, we obtain
\begin{equation}\label{eq:Talpha-def}
  T_{\alpha}(\epsilon,c)\;=\;\int_{\epsilon}^{c}\frac{1}{\alpha(y)}\,\mathrm{d}y.
\end{equation}
Define its windowed \emph{nominal rate}
\begin{equation}\label{eq:sigma-alpha}
  \sigma_{\alpha}(\epsilon,c)\;\doteq\;\frac{\ln(c/\epsilon)}{T_{\alpha}(\epsilon,c)}.
\end{equation}
By the comparison principle \eqref{eq:intro-V-ineq}–\eqref{eq:intro-ub} (which yields $V(x(t))\le y(t)$ for all $t\ge0$), the comparison trajectory reaches $\epsilon$ no earlier than the true trajectory:
\begin{equation}\label{eq:T-ineq}
  T_{\alpha}(\epsilon,c)\ge T(\epsilon,c)\quad\iff\quad
  \sigma_{\alpha}(\epsilon,c)\le\sigma_{\mathrm{nom}}(\epsilon,c).
\end{equation}
Thus, $\sigma_{\alpha}(\epsilon,c)$ is a \emph{conservative} (certified) estimate of the true nominal rate on $[\epsilon,c]$.

\subsubsection*{Linear case}
For the linear comparison $\alpha(V)=\sigma V$,
\[
  T_{\alpha}(\epsilon,c)=\frac{1}{\sigma}\ln\!\frac{c}{\epsilon},
  \qquad
  \sigma_{\alpha}(\epsilon,c)\equiv\sigma\quad\text{(for any window $[\epsilon,c]$).}
\]
Unlike the strictly pointwise constraint $\dot V\le -\sigma V$ (constant rate at every level), the windowed metric \eqref{eq:sigma-alpha} aggregates the guaranteed decay \emph{over the window}, allowing structural reshaping of $\alpha$ (e.g., concavity) under a fixed endpoint cap to trade interior decay against feasibility. For each two adjacent subintervals $[\epsilon,c_0], [c_0,c]$, the relation of the related nominal rate is presented in \eqref{eq:rate-fraction}.
\begin{lem}[Composition of nominal rates]\label{cor:rate-composition}
Consider $0<\epsilon<c_0<c$ and a class-$\mathcal K$ comparison function $\alpha$. Define
\[\sigma_1\doteq \sigma_\alpha(\epsilon,c_0),  
\sigma_2 \doteq \sigma_\alpha(c_0,c), 
\lambda \doteq {\ln(c/c_0)}/{\ln(c/\epsilon)} \in (0,1).\]
Then, the nominal rate on $[\epsilon,c]$ satisfies:
\begin{subequations}
     \begin{align}
        \label{eq:rate-harmonic}
\frac{1}{\sigma_\alpha(\epsilon,c)}
\;=\; \frac{1-\lambda}{\sigma_1} \;+\; \frac{\lambda}{\sigma_2}
    \end{align}
\begin{equation}\label{eq:rate-fraction}
\sigma_\alpha(\epsilon,c)
\;=\;
\frac{\sigma_1\,\sigma_2}{\,\lambda\,\sigma_1 + (1-\lambda)\,\sigma_2\,}.
\end{equation}
\begin{equation}\label{eq:rate-bounds}
\min\{\sigma_1,\sigma_2\}
\;\le\;
\sigma_\alpha(\epsilon,c)
\;\le\;
\max\{\sigma_1,\sigma_2\},
\end{equation}
\end{subequations}
with equality in \eqref{eq:rate-bounds} if and only if $\sigma_1=\sigma_2$ (since $\lambda\in(0,1)$).
\end{lem}
\begin{proof}
By definition,
\(
T_\alpha(\epsilon,c)
= T_\alpha(\epsilon,c_0)+T_\alpha(c_0,c),
\)
\[
\frac{1}{\sigma_\alpha(\epsilon,c)}
=\frac{T_\alpha(\epsilon,c)}{\ln(c/\epsilon)}
=\frac{T_\alpha(\epsilon,c_0)}{\ln(c/\epsilon)}
 +\frac{T_\alpha(c_0,c)}{\ln(c/\epsilon)}
=\frac{1-\lambda}{\sigma_1}+\frac{\lambda}{\sigma_2},
\]
which is \eqref{eq:rate-harmonic}; \eqref{eq:rate-fraction} is algebraically rearranged. Since \eqref{eq:rate-harmonic} is a weighted harmonic mean of $\sigma_1,\sigma_2$ with weights $1-\lambda$ and $\lambda$ in $(0,1)$, \eqref{eq:rate-bounds} follows, with equality iff $\sigma_1=\sigma_2$.
\end{proof}

This lemma is useful for shaping the nominal 
rate: it allows one to maintain a desired overall rate while relaxing the decay constraint at high Lyapunov levels to preserve feasibility under input bounds.

\subsection{Induced endpoint cap under input bounds}\label{subsec:feasibility}
Consider a smooth CLF $V$. Define the pointwise  decay cap
\begin{equation}\label{eq:Dmax_def}
  D_{\max}(x,\theta)\;\doteq\;\max_{u\in\mathbb{U}(\theta)}\big\{-L_f V(x)-L_g V(x)\,u\big\}.
\end{equation}
For $\mathbb{U}(\theta)=\{u:\|u\|_\infty\le\theta\}$, linear optimization over  $u$ gives
\begin{align}
     D_{\max}(x,\theta)=-L_f V(x)+\theta\,\|L_g V(x)\|_1.
\end{align}
At point $x$, an admissible decay $\alpha(V(x))$ should satisfy
\begin{align}\label{eq: pointwise condition}
    \alpha\big(V(x)\big)\;\le\;D_{\max}(x,\theta),
\end{align}
which represents the pointwise feasibility condition. The sublevel set $\Omega(V,v)$ induces a level-wise cap 
\begin{equation}\label{eq:alpha_cap}
  \bar{\alpha}(\theta,v)\doteq \sup_{x\in\Omega(V,v)} D_{\max}(x,\theta), \quad 0<v\leq c_{\max}.
\end{equation}
which yields the necessary condition at level $V=v$:
\begin{align}\label{eq:level-wise condition}
    \alpha(v)\leq \bar{\alpha}(\theta,v).
\end{align}
Since $\Omega(V,v)\subseteq\Omega(V,c)$ for $v\le c$, $\bar{\alpha}(\theta,v)\le \bar{\alpha}(\theta,c)$. Hence,
\[\forall 0<V\leq c,\qquad\alpha(V)\;\le\;\alpha(c)\;\le\;\bar{\alpha}(\theta,c).
\]
Thus, under input bounds, on any window $[\epsilon,c]$ the admissible $\alpha(V)$ is capped by the \emph{endpoint cap} $\bar{\alpha}(\theta,c)$, motivating shape design of $\alpha$ under a fixed endpoint.
\subsubsection*{Global exponential rate}
For the linear comparison $\alpha(V)=\sigma V$, the endpoint condition $\alpha(c)=\sigma c\le \bar{\alpha}(\theta,c)$ gives
\[
  \sigma\;\le\;{\bar{\alpha}(\theta,c)}/{c}\doteq\sigma_c.
\]
Actuator limits therefore impose an endpoint cap on the achievable global exponential rate. Any $\sigma>\sigma_c$ violates the necessary cap at $V=c$ for which the decay inequality is infeasible. Consequently, uniform feasibility over $\Omega(V,c)$ cannot be guaranteed.

\subsection{Comparison functions}
We mainly focus on class-$\mathcal{K}$ comparison functions. According to shape, we classify $\alpha\in\mathcal{K}$ into four classes: linear $\mathcal{K}_l$, strictly concave $\mathcal{K}_{cave}$, strictly convex $\mathcal{K}_{vex}$, and mixed
\[\mathcal{K}_{mix}\doteq\mathcal{K}\setminus(\mathcal{K}_l\cup\mathcal{K}_{cave}\cup\mathcal{K}_{vex}).\]
Note that any $\alpha\in\mathcal{K}$ admits a quasi-linear representation:
\begin{align}
    \alpha(y)= s(y)y\label{eq:alpha_quasilinear}
\end{align}
 where $s:\mathbb{R}_{\geq0}\to\mathbb{R}_{>0}$ denotes the dynamic (scaling) factor defined as:
 \begin{align}\label{de:dynamic factor}
s(y)\doteq\begin{cases}
     \frac{\alpha(y)}{y},\qquad y>0,\\
\underset{{y\to0_+}}{\lim}\frac{\alpha(y)}{y}\in[0,\infty],\quad y=0.
 \end{cases}     
 \end{align}
Here we leverage the monotonic properties of the dynamic factor about the shape of $\alpha$.
\begin{lem}[Dynamic-factor characterization]\label{Lem monotonicity of s}
    Suppose $\alpha\in\mathcal{K}$ with dynamic factor $s$. The following statements hold:
    \begin{itemize}
        \item[1)] $s(y)$ is a constant if $\alpha\in\mathcal{K}_l$;
        \item[2)] $s(y)$ is strictly decreasing in $y$ if $\alpha\in\mathcal{K}_{cave}$;
        \item[3)] $s(y)$ is strictly increasing in $y$ if $\alpha\in\mathcal{K}_{vex}$.
    \end{itemize}
\end{lem}
\begin{proof}
    The linear case is obvious. By the definition of strictly concave, for $0<\lambda<1$ and $y>0$, we have 
    \begin{align*}
        \alpha(\lambda y+(1-\lambda)0)>\lambda\alpha(\lambda y)+(1-\lambda)\alpha(0)
    \end{align*}
    which implies $s(\lambda y)>s(y)$ such that $s(y)$ is decreasing. Similarly, we can prove if $\alpha(y)$ is convex, $s(y)$ is increasing.
\end{proof}

These properties will be used for analysis and construction. 
\begin{defi}[Concave Factor]\label{def:concave factor}
    For $\alpha\in\mathcal{K}_{cave}$, the associated $s$ in \eqref{eq:alpha_quasilinear} is called a concave factor that  concavifies $y$.
\end{defi}

The concave factor serves as a practical design handle for constructing Lipschitz class-$\mathcal{K}_{\mathrm{cave}}$ comparison functions.

\section{Concave Comparison Functions for Constrained Upper-Bound Decay}
\label{sec:concave}
This section quantifies how the \emph{shape} (concave, linear, convex) of a class-\(\mathcal{K}\) comparison function \(\alpha\) governs the decay of the upper-bound system \eqref{eq:intro-ub} on the evaluation window \([\epsilon, c]\), subject to the endpoint cap \(\alpha(c)\le \bar\alpha(\theta,c)\). Unless stated otherwise, “concave’’ and “convex’’ mean strictly concave/convex.

\subsection{Decay ordering and acceleration}
We mainly compare three classes $\mathcal{K}_l$, $\mathcal{K}_{cave}$, and $\mathcal{K}_{vex}$ under the same
endpoint value $\alpha(c)$, and claim that $\mathcal{K}_{mix}$ functions lie in  between. As a reminder, a class-$\mathcal K$ function is continuous, strictly increasing,
and satisfies $\alpha(0)=0$; hence it is locally Lipschitz on $[\epsilon,c]$. We write $\alpha_l(y)=\sigma y\in\mathcal{K}_l$, $\alpha_{{cave}}\in\mathcal{K}_{{cave}}$,
and $\alpha_{{vex}}\in\mathcal{K}_{{vex}}$. 

\begin{prop}[Ordering under equal endpoint cap]\label{prop:ordering}
Let $\alpha_l(c)=\alpha_{cave}(c)=\alpha_{vex}(c)$.
Then for any $\epsilon\in(0,c)$,
\[
  \sigma_{\alpha_{vex}}(\epsilon,c)\;<\;\sigma_{\alpha_l}(\epsilon,c)\;=\;\frac{\alpha_l(c)}{c}
  \;<\;\sigma_{\alpha_{cave}}(\epsilon,c),
\]
where $\sigma_{\alpha_{(\cdot)}}(\epsilon,c)$ denote the nominal rate defined in \eqref{eq:sigma-alpha}.
\end{prop}
\begin{proof}
    Let $s_{l},s_{cave},s_{vex}$ be the dynamic factors as defined in \eqref{de:dynamic factor} of $\alpha_l,\alpha_{cave},\alpha_{vex}$, respectively. Then, $s_{cave}(c)>s_{l}(c)>s_{vex}(c)$. By Lem.\ref{Lem monotonicity of s}, for $y\in[\epsilon,c)$, $s_{cave}(y)>s_{l}(y)>s_{vex}(y)$, such that $T_{\alpha_{cave}}(\epsilon,c)<T_{\alpha_l}(\epsilon,c)<T_{\alpha_{vex}}(\epsilon,c)$. Equivalently, we have $\sigma_{\alpha_{vex}}(\epsilon,c)<\sigma_{\alpha_l}(\epsilon,c)<\sigma_{\alpha_{cave}}(\epsilon,c)$.
\end{proof}

Figure~\ref{fig:sketch_of_comparison} depicts the equal–endpoint case at $y=c$. By geometry,
a strictly concave $\alpha_{{cave}}$ lies above its chord while a strictly convex
$\alpha_{{vex}}$ lies below; hence, for all $y\in(0,c)$,
\[
\alpha_{{cave}}(y)\;>\;\alpha_{l}(y)\;>\;\alpha_{{vex}}(y),
\] which supports the above proposition. Moreover, concave comparisons can achieve a faster windowed decay even when the endpoint value is reduced.

\begin{thm}[Acceleration under a relaxed endpoint cap]
\label{thm:relaxed-endpoint-accel}
Let $\alpha_l(y)=\sigma y$ be the linear baseline on $(0,c]$.
\noindent\emph{(i)} 
If $\alpha_{vex}\in\mathcal{K}_{vex}$ (strictly convex) and $\alpha_{vex}(c)\le \alpha_l(c)$, then for every $\epsilon\in(0,c)$, $ \sigma_{\alpha_{vex}}(\epsilon,c)\ <\ \sigma$.
\noindent\emph{(ii)}
Suppose $\alpha_{cave}\in\mathcal{K}_{cave}$ (strictly concave) with
$\alpha_{cave}(c)<\alpha_l(c)$ and there exists $y^\star\in(0,c)$ such that
$\alpha_{cave}(y^\star)>\alpha_l(y^\star)$, then there exists $c_0\in(y^\star,c)$
and $\epsilon_0\in(0,c_0)$ such that for all $\epsilon\in(0,\epsilon_0)$,
  $\sigma_{\alpha_{cave}}(\epsilon,c)\ >\ \sigma$.
\end{thm}

\begin{proof}
Write $\alpha(y)=s(y)\,y$ with the dynamic factor $s(y)=\alpha(y)/y$. \emph{(i) Convex case.} If $\alpha_{vex}$ is strictly convex, then $s_{vex}$ is strictly
increasing on $(0,c]$. Hence $s_{vex}(y)\le s_{vex}(c)$ for all $y\in(0,c]$, and
\[
  T_{\alpha_{vex}}(\epsilon,c)
  \;=\;\int_\epsilon^c \frac{dy}{s_{vex}(y)\,y}
  \;\ge\;\frac{1}{s_{vex}(c)}\int_\epsilon^c \frac{dy}{y}
  \;=\;\frac{\ln(c/\epsilon)}{s_{vex}(c)}.
\]
Since $\alpha_{vex}(c)\le \alpha_l(c)=\sigma c$, we have $s_{vex}(c)\le\sigma$, hence
$T_{\alpha_{vex}}(\epsilon,c)\ge \frac{\ln(c/\epsilon)}{\sigma}=T_{\alpha_l}(\epsilon,c)$.
Strict convexity makes the inequality strict, so
$\sigma_{\alpha_{vex}}(\epsilon,c)<\sigma$. \emph{(ii) Concave case.} If the condition holds, then
$s_{cave}(y^\star)>\sigma$ and $s_{cave}(c)<\sigma$. Because $s_{cave}$ is strictly
decreasing and continuous, there exists a unique $c_0\in(y^\star,c)$ with
$s_{cave}(c_0)=\sigma$, i.e., $\alpha_{cave}(c_0)=\alpha_l(c_0)$.
Over $(0,c_0)$ one has $s_{cave}>\sigma$, hence
$\sigma_{\alpha_{cave}}(\epsilon,c_0)>\sigma$ for all $\epsilon\in(0,c_0)$; over
$(c_0,c)$ one has $s_{cave}<\sigma$, hence $\sigma_{\alpha_{cave}}(c_0,c)<\sigma$. Using the two-interval compositions (cf. ~\eqref{eq:rate-fraction})
\[
  \sigma_{\alpha}(\epsilon,c)
  = \frac{\sigma_{\alpha}(\epsilon,c_0)\,\sigma_{\alpha}(c_0,c)}
         {\lambda\,\sigma_{\alpha}(\epsilon,c_0)+(1-\lambda)\,\sigma_{\alpha}(c_0,c)},
   ~\lambda=\frac{\ln(c/c_0)}{\ln(c/\epsilon)},
\]
we see that as $\epsilon\to 0$ (so $\lambda\to 0$),
$\lim_{\epsilon\to0^+}\sigma_{\alpha_{cave}}(\epsilon,c)= \sigma_{\alpha_{cave}}(0^+,c_0)>\sigma$. Using monotonicity, one can simply pick $\epsilon_0$ with $\sigma_{\alpha_{{cave}}}(\epsilon_0,c)>\sigma$; then for any smaller $\epsilon$ the inequality still holds.
\end{proof}

Proposition~\ref{prop:ordering} and Theorem~\ref{thm:relaxed-endpoint-accel} together show that
 class-$\mathcal{K}_{cave}$ comparison functions enable acceleration with \emph{endpoint relaxation}:
for suitable windows $[\epsilon,c]$, one can have
\[
\alpha_{cave}(c)\;<\;\alpha_l(c)\quad\text{while}\quad \sigma_{\alpha_{cave}}(\epsilon,c)\;>\;\sigma={\alpha_l(c)}/{c}.
\]

\subsubsection*{Mixed shapes via window decomposition}
For $\alpha_{{mix}}\in\mathcal{K}_{{mix}}$ we partition the window
$[\epsilon,c]$ into subwindows
\[
\epsilon=a_0<a_1<\cdots<a_N=c,
\]
on each of which $\alpha_{{mix}}$ is (i) strictly concave, (ii) linear, or (iii) strictly convex.
For $y\in[a_{i-1},a_i]$, denote the chord (secant) of $\alpha_{{mix}}$ by
\[
\chi_i(y)\doteq
\frac{\alpha_{{mix}}(a_i)-\alpha_{{mix}}(a_{i-1})}{a_i-a_{i-1}}(y-a_{i-1})
+\alpha_{{mix}}(a_{i-1}).
\]
By additivity of the crossing-time integral,
\[
T_{\alpha_{{mix}}}(\epsilon,c)
=\sum_{i=1}^N \int_{a_{i-1}}^{a_i}\frac{1}{\alpha_{{mix}}(y)}\,dy.
\]
On a concave subwindow, $\alpha_{{mix}}(y)>\chi_i(y)$ and the local integral
$\int_{a_{i-1}}^{a_i}\!\frac{1}{\alpha_{{mix}}}\,dy$ is \emph{reduced} relative to the chord;
on a linear subwindow, $\alpha_{{mix}}=\chi_i$ (neutral);
on a convex subwindow, $\alpha_{{mix}}(y)<\chi_i(y)$ and the local integral is \emph{increased}. Consequently,
only strictly concave subwindows contribute to acceleration (relative to the chord),
linear subwindows are neutral, and convex subwindows degrade it.

\subsection{Necessity of concavity under an endpoint cap}
To further study comparison functions under the endpoint constraint,
we introduce the following metric.
\begin{defi}[Endpoint-Relaxation ratio]\label{def:RR}
Define
\begin{align}\label{eq: relaxation ratio}
    r_{\alpha}(\epsilon,c)\doteq \frac{\alpha(c)}{\sigma_{\alpha}(\epsilon,c)\,c}, ~0<\epsilon<c,
\end{align}
where $\sigma_{\alpha}(\epsilon,c)$ is defined in \eqref{eq:sigma-alpha}.
\end{defi}

Interpretation: $r_\alpha(\epsilon,c)<1$ (relaxation) reflects same rate with a \emph{smaller} endpoint than linear, $r_\alpha=1$ represents same rate with the \emph{same} endpoint as linear, and   $r_\alpha>1$ means achieving the same rate requires a \emph{larger} endpoint than linear.
For the linear comparison $\alpha_l$, $r_{\alpha_l}\equiv 1$ on any window. 

\begin{prop}[Ordering for endpoint-relaxation ratio]\label{prop:RR}
For any $[\epsilon,c]\subset(0,\infty)$, we obtain
\[\quad r_{\alpha_{cave}}(\epsilon,c)<1=r_{\alpha_l}(\epsilon,c)<r_{\alpha_{vex}}(\epsilon,c)\]
whenever $\alpha_{cave}\in\mathcal{K}_{cave}$ and $\alpha_{vex}\in\mathcal{K}_{vex}$. 
\end{prop}
\begin{proof}  Linear case is immediate. For $\alpha_{cave}\in\mathcal{K}_{cave}$, \[\alpha_{cave}(y)> {\alpha_{cave}(c)y}/{c} ~\text{holds}, ~\text{for all } y\in(0,c), ~\text{which yields}\] 
\[T_{\alpha_{cave}}(\epsilon,c)< \int_\epsilon^c \frac{c\,dy}{\alpha_{cave}(c)\,y}=\frac{c}{\alpha_{cave}(c)}\ln\left(\frac{c}{\epsilon}\right).\]
Substituting it into \eqref{eq: relaxation ratio}, one has $r_{\alpha_{cave}}< 1$. The strictly convex case reverses the
inequalities, giving $r_{\alpha_{vex}}>1$. 
\end{proof}

\subsubsection*{Endpoint relaxation and concave shaping}
The endpoint–relaxation ratio is \emph{intrinsic} to a comparison function on a given window and
is invariant under positive scaling: for any $k>0$, $r_{k\alpha}(\epsilon,c)=r_{\alpha}(\epsilon,c)$.
Hence we may take any strictly concave class-\(\mathcal K\) template \(\alpha_{{cave}}\) and
normalize/scale it to obtain a comparison function that simultaneously relaxes the endpoint and
accelerates decay (cf. Sec.~\ref{sec:construct}).

Let the linear baseline be \(\alpha_l(y)=\sigma y\) (so \(\alpha_l(c)=\sigma c\)), and write
\(\alpha_{{cave}}(y)=s_{{cave}}(y)\,y\) with strictly decreasing
\(s_{{cave}}\) (the concave factor). Define
\begin{equation}\label{eq:concave-normalized}
  \alpha(y)\;=\;\frac{s_{\mathrm{cave}}(y)}{s_{\mathrm{cave}}(c)}\,r\,\sigma\,y,
  \qquad r_{\alpha_{\mathrm{cave}}}(\epsilon,c)\;<\;r\;<\;1.
\end{equation}
Then \(\alpha(c)=r\,\sigma c\) (endpoint relaxation), and using
\(\sigma_{k\alpha}=k\,\sigma_{\alpha}\) together with \(r>r_{\alpha_{\mathrm{cave}}}(\epsilon,c)\)
yields $\sigma_{\alpha}(\epsilon,c)>\sigma$,
i.e., faster windowed decay with a reduced endpoint value. 

\begin{lem}[Concavity is necessary for relaxation]\label{lem:necessary-concave}
Let $\alpha\in\mathcal{K}$ be continuous on $[0,c]$ and $C^2$ on $(0,c)$.
If $r_\alpha(\epsilon,c)<1$ for some $0<\epsilon<c$, then there exists a nontrivial interval
$\mathcal{I}_{\mathrm{cave}}\subset(0,c)$ on which $\alpha$ is strictly concave.
\end{lem}

\begin{proof}
Let $\chi(y)\doteq\frac{\alpha(c)}{c}\,y$ be the chord through $(0,0)$ and $(c,\alpha(c))$.
The condition $r_\alpha(\epsilon,c)<1$ is equivalent to
\[
T_\alpha(\epsilon,c)=\int_{\epsilon}^{c}\frac{1}{\alpha(y)}dy
<
\int_{\epsilon}^{c}\frac{1}{\chi(y)}dy
=\frac{c}{\alpha(c)}\ln\frac{c}{\epsilon}.
\]
Hence 
\(
\int_{\epsilon}^{c}\big(\tfrac{1}{\alpha(y)}-\tfrac{1}{\chi(y)}\big)dy<0,
\)
which implies that
\(
\tfrac{1}{\alpha(y)}-\tfrac{1}{\chi(y)}<0
\)
on a set of positive measure in $(\epsilon,c)$. Since $\alpha,\chi>0$ on $(\epsilon,c)$, this yields
$\alpha(y)>\chi(y)$ on such a set. Define $h(y):=\alpha(y)-\chi(y)$. Then $h$ is continuous on $[0,c]$, $C^2$ on $(0,c)$, and satisfies
$h(0)=h(c)=0$ while $h>0$ on a set of positive measure in $(\epsilon,c)$. 
If $h''(y)\ge 0$ for all $y\in(0,c)$, then $h$ would be convex, and with $h(0)=h(c)=0$ convexity would imply
$h(y)\le 0$ for all $y\in(0,c)$ contradicting the condition. Therefore there exists $y_0\in(0,c)$ with
$h''(y_0)<0$. By continuity of $h''$, there is a nontrivial open interval 
$\mathcal{I}_{\mathrm{cave}}\ni y_0$ on which $h''<0$, i.e., $\alpha''<0$. Thus $\alpha$ is strictly concave
on $\mathcal{I}_{\mathrm{cave}}$.
\end{proof}

These results show that strictly concave comparison functions ($\mathcal{K}_{\mathrm{cave}}$)
can \emph{reduce the endpoint value} while still certifying \emph{faster windowed decay}.
The reduced endpoint directly lowers the required actuation level (e.g., peak input $\|u\|_\infty$),
as analyzed in the next section. Moreover, for a mixed-shape comparison
$\alpha\in\mathcal{K}_{\mathrm{mix}}$, achieving endpoint relaxation on a window necessarily
requires \emph{strict concavity on a nontrivial subinterval} of $(0,c)$
(cf.\ Lemma~\ref{lem:necessary-concave}).

\section{Lyapunov Decay in the CLF Framework}\label{sec:exact-CLF}
This section links the upper–bound decay analysis in Sec.\ref{sec:concave}
to the feasibility of the CLF derivative inequality \eqref{eq:V-ineq} under input bounds. We show that the concave comparison functions reduce the lower-bound actuation level while guaranteeing faster decay.

\subsection{Required Actuation Level for a Desired Decay}\label{subsec:theta-demand}
We characterize the trade–off between a target decay constraint $(\dot V(x,u)\leq-\alpha(V(x)))$ and the required actuation level. Throughout, we make the following regular assumption.
\begin{Ass}[Regularity]\label{ass:regularity}
There exist $L_1,L_2>0$, and $\{\bar g_i\}_{i=1}^m$ such that, for all $x\in\mathbb{X}$, $\|f(x)\|\leq L_1\|x\|$, and
\[
  \|\nabla V(x)\|\ \le\ L_2 \|x\|, 
  \qquad 
  \|g_i(x)\|\ \le\ \bar g_i,\ \ i=1,\dots,m,
\]
where $g(x)=[g_1(x),\dots,g_m(x)]$.  
\end{Ass}

\begin{defi}[Actuation level on a sublevel set]
For a given $\alpha\in\mathcal{K}$ and $v>0$, the required actuation level (minimal peak actuation $\|u\|_{\infty}$) to enforce
$\dot{V}(x,u)\le-\alpha(V)$ uniformly on $\Omega(V,v)$ is defined as:
\begin{equation}\label{eq:theta-min-def}
  \theta_{\min}(\alpha;v)
  \doteq
  \inf_{\theta\ge0}\Big\{\theta: \forall x\in\Omega(V,v),~\alpha\big(V(x)\big)\le D_{\max}(x,\theta)\Big\}.
\end{equation}
\end{defi}

Using $D_{\max}(x,\theta)=-L_fV(x)+\theta\,\|L_gV(x)\|_1$, \eqref{eq:theta-min-def} admits the closed form
\begin{equation}\label{eq:theta-min-closed}
  \theta_{\min}(\alpha;v)
  \;=\;
  \sup_{x\in\Omega(V,v)}
  \left[
    \frac{\alpha\big(V(x)\big)+L_fV(x)}{\ \|L_gV(x)\|_1\ }
  \right]_+
\end{equation}
where $[z]_+ \doteq \max\{z,0\}$. When $\alpha(V(x))+L_fV(x)\leq0$ for all $x\in\Omega(V,v)$, then $\theta_{\min}(\alpha;v)=0$. 
If there exists $x\in\Omega(V,v)$ with $x\ne0$, $\|L_gV(x)\|_1=0$ and $\alpha(V(x))+L_fV(x)>0$, then $\theta_{\min}(\alpha;v)=\infty$. 
\subsubsection{Small-$x$ behavior and finiteness}
We quantify when the actuation level $\theta_{\min}(\alpha;v)$ remains finite (or vanishes) as $v\to0$.
Near the origin assume, $1\le q_1\le 2$,
\[
  L_fV(x)=\mathcal{O}(\|x\|^2),\quad 
  \alpha(V(x))=\mathcal{O}(\|x\|^{q_1}),
\]
and impose a \emph{lower} growth bound on the input channel:
\begin{equation}\label{eq:lg-lower}
 \forall \|x\|<\rho, ~\|L_gV(x)\|_1 \ge c_g\|x\|^{q_2},
  \ c_g>0, \ q_2\ge 1.
\end{equation}

From \eqref{eq:theta-min-closed},
the numerator is $\mathcal{O}(\|x\|^{\min\{q_1,2\}})=\mathcal{O}(\|x\|^{q_1})$, and by \eqref{eq:lg-lower}
the denominator is $\mathcal{O}(\|x\|^{q_2})$. Therefore 
\[
  \left[\frac{\alpha(V(x))+L_fV(x)}{\|L_gV(x)\|_1}\right]_+
  \;=\; \mathcal{O}\!\left(\|x\|^{\,q_1-q_2}\right)_+,
\]
and we obtain the following situations as $v\to0$:
\begin{itemize}
  \item \textbf{Vanishing level} if $q_2<q_1$:  $q_1-q_2>0$ and
  $\lim_{v\to0}\theta_{\min}(\alpha;v)=0$.
  \item \textbf{Finite, nonzero limit (or bounded)} if $q_2=q_1$: the ratio remains bounded (approaches a constant
  determined by leading coefficients).
  \item \textbf{Blow-up} if $q_2>q_1$ and the numerator is positive near the origin
  (i.e., $\alpha(V(x))+L_fV(x)>0$ on a punctured neighborhood):
   $q_1-q_2<0$ and $\lim_{v\to0}\theta_{\min}(\alpha;v)=\infty$, e.g., scalar system $\dot{x}=x^qu$ with a CLF $V=\frac{1}{2}x^2$, if $q>1$, any linear comparison leads to infinite actuation level near the origin.
  If the numerator becomes nonpositive sufficiently close to the origin, the positive part yields $0$ instead.
\end{itemize}

\noindent\textbf{Examples.}
(i) \emph{Linear comparison} $\alpha(V)=\sigma V$ gives $q_1=2$.
If $\|L_gV(x)\|_1\ge c_g\|x\|$ ($q_2=1$, e.g., constant $g$ with full actuation and quadratic $V=x^\top Px$), then
$\theta_{\min}(\alpha;v)\to 0$ as $v\to0$.
\noindent
(ii) \emph{Root-type} $\alpha(V)=C V^{p}$ with $p\in(0,1)$ has $q_1=2p$.
If $q_2=1$, finiteness near the origin requires $2p\ge 1$ (i.e., $p\ge \tfrac12$); for $p<\tfrac12$ the level
blows up, reflecting the stiffness of non-Lipschitz rates at the origin. 

\begin{prop}[Scaling law: $\theta_{\min}$ vs.\ $\xi\alpha$]\label{prop:theta-scaling-law}
Consider $v>0$ and $\alpha\in\mathcal{K}$, and assume $\theta_{\min}(\alpha;v)<\infty$. For $\xi\ge0$, define
\[
  \theta_{\min}(\xi\alpha;v)
  \;=\;
  \sup_{x\in\Omega(V,v)}
  \left[
    \frac{\xi\,\alpha(V(x)) + L_fV(x)}{\|L_gV(x)\|_1}
  \right]_+ .
\]
Then for any $0\le\xi_1<\xi_2<\infty$,
\[
  \theta_{\min}(\xi_1\alpha;v)\ \le\ \theta_{\min}(\xi_2\alpha;v).
\]
Moreover, if $\theta_{\min}(\xi_1\alpha;v)>0$, then
\[
  \theta_{\min}(\xi_2\alpha;v)\;>\;\theta_{\min}(\xi_1\alpha;v).
\]
Equivalently, if $\theta_{\min}(\alpha;v)>0$, scaling down ($0<\xi<1$) strictly decreases the level, while scaling up ($\xi>1$) strictly increases it.
\end{prop}

\begin{proof}
For each $x\in\Omega(V,v)$, let
\[
  \psi_x(\xi)\;\doteq\;\left[\frac{\xi\,\alpha(V(x))+L_fV(x)}{\|L_gV(x)\|_1}\right]_+ .
\]
If $\|L_gV(x)\|_1>0$, then $\psi_x(\xi)$ is the positive part of an affine function in $\xi$ with slope
$\alpha(V(x))/\|L_gV(x)\|_1\ge0$, hence it is nondecreasing (indeed convex). If $\|L_gV(x)\|_1=0$, finiteness of $\theta_{\min}(\alpha;v)$ forces the numerator $\le0$, so $\psi_x\equiv0$, which is also nondecreasing. Thus $\theta_{\min}(\xi\alpha;v)=\sup_{x\in\Omega(V,v)}\psi_x(\xi)$
is nondecreasing in $\xi$, proving $\theta_{\min}(\xi_1\alpha;v)\le\theta_{\min}(\xi_2\alpha;v)$ for $\xi_2>\xi_1$. If $\theta_{\min}(\xi_1\alpha;v)>0$, there exists $x^\star$ with $\psi_{x^\star}(\xi_1)>0$. Then necessarily
$\alpha(V(x^\star))>0$ and $\|L_gV(x^\star)\|_1>0$, so $\psi_{x^\star}$ has strictly positive slope at $\xi_1$.
Hence $\psi_{x^\star}(\xi_2)>\psi_{x^\star}(\xi_1)$ for any $\xi_2>\xi_1$,  yielding
$\theta_{\min}(\xi_2\alpha;v)>\theta_{\min}(\xi_1\alpha;v)$.
\end{proof}

Therefore, under the linear comparison $\alpha(V)=\sigma V$, the required actuation level is nondecreasing in $\sigma$. In fact, if $\theta_{\min}(\sigma_1 V;c)>0$ and $\sigma_2>\sigma_1$, then
\[
\theta_{\min}(\sigma_2 V;c)\;>\;\theta_{\min}(\sigma_1 V;c),
\]
so achieving a larger certified exponential rate necessarily requires a greater actuation level (larger $\|u\|_{\infty}$).

\subsubsection{Endpoint–dominated lower bounds (level-wise growth)}
To capture how the required input grows with the Lyapunov level, note that for $x\in\Omega(V,v)$, by Assumption. \ref{ass:regularity}, \[\|f(x)\|\le L_1\|x\|,\quad\|\nabla V(x)\|\le L_2\|x\|,\quad\|g_i(x)\|\le\bar g_i,\]
 and the CLF bounds $k_1\|x\|^2\le V(x)$, we have
\[
  L_fV(x)\leq L_1L_2\|x\|^2 \ \le \frac{L_1L_2}{k_1}\,V \;\doteq\; k_3\,V, \quad\text{and}
\]
\begin{align*}
    &\|L_gV(x)\|_1=\sum_i|\nabla V^\top g_i| \le \|\nabla V\|\sum_i\|g_i\|
  \\ &\quad\le L_2\left(\sum_{i=1}^m \bar g_i\right)\|x\|
  \le\ \frac{L_2}{\sqrt{k_1}}\!\left(\sum_{i=1}^m \bar g_i\right)\!\sqrt{V}
  \;\doteq\; k_4\,\sqrt{V}.
\end{align*} 
Therefore, we obtain a level-wise lower bound:
\begin{align}\label{eq: theta ubar}
     \theta_{\min}(\alpha;v)
   \ge
  \sup_{0<V\le v}
  \left[\frac{\alpha(V)-k_3 V}{k_4\sqrt{V}}\right]_+ \doteq\ubar{\theta}_{\min}(\alpha;v).
\end{align}
Using $\theta\geq\theta_{\min}({\alpha;v})$, \eqref{eq: theta ubar} gives the necessary level-wise condition (cf. \eqref{eq:level-wise condition})
\begin{align}
    \alpha(V)\leq\bar{\alpha}(\theta,V)\leq k_3V+\theta k_4\sqrt{V}
\end{align}

\begin{prop}[Endpoint-dominated lower bound]\label{prop:theta-lb-endpoint}
For $\alpha\in\mathcal{K}_l$, $\ubar{\theta}_{\min}(\alpha;v)$ is endpoint-dominated that
\begin{align}\label{eq: endpoint dominated}
  0<V\leq v,\quad \ubar{\theta}_{\min}(\alpha;v) = \left[\frac{\alpha(v)-k_3 v}{k_4\sqrt{v}}\right]_+.  
\end{align}
\end{prop}
\begin{proof}
    Write $\alpha(V)=\sigma V$. If $\sigma\leq k_3$, then $\ubar{\theta}_{\min}(\alpha;v)\equiv0$ such that \eqref{eq: endpoint dominated} holds. If $\sigma>k_3$, for $V\in(0,v]$, function
    \begin{align*}
        \frac{\alpha(V)-k_3 V}{k_4\sqrt{V}}=  \frac{\sigma -k_3}{k_4}\sqrt{V}
    \end{align*}
    is strictly increasing such that the maximum value is obtained at $V=v$. 
\end{proof}

Both Propositions~\ref{prop:theta-scaling-law} and \ref{prop:theta-lb-endpoint} show that, for the linear comparison, lowering the endpoint value $\alpha(c)$ is necessary to reduce the  actuation level
$\theta_{\min}$ on $x\in\Omega(V,c)$ ( to lower $\|u\|_\infty$).

\begin{prop}[Actuation savings via endpoint relaxation]\label{prop:actuation-savings}
Consider a window $[\epsilon,c]$ and target rate $\sigma>k_3$. Let the linear baseline be $\alpha_l(V)=\sigma V$ and 
$\alpha_d(V)=s(V)\sigma V\in\mathcal{K}$ satisfy $\sigma_{\alpha_d}(\epsilon,c)>\sigma$.

\noindent\emph{(i) Necessity of endpoint relaxation and concavity:}
\[\ubar{\theta}_{\min}(\alpha_d;c)<\ubar{\theta}_{\min}(\alpha_l;c)\quad\Longrightarrow \quad s(c)<1 ~(\alpha_d(c)<\alpha_l(c)).\]
Moreover, $\alpha_d$ must be \emph{strictly concave} on a nontrivial subinterval $\mathcal{I}_{cave}\subset(0,c)$.

\smallskip
\noindent\emph{(ii) No savings with linear or convex shapes.}
If $\sigma_{\alpha_d}(\epsilon,c)>\sigma$, \[\forall\alpha_d\in\{\mathcal{K}_l,\mathcal{K}_{{vex}}\}, \qquad\ubar{\theta}_{\min}(\alpha_d;c)>\ubar{\theta}_{\min}(\alpha_l;c).\]
\end{prop}

\begin{proof}
\emph{(i)} Let $\varphi(\alpha;v)\doteq([\alpha(v)-k_3v]_+)/(k_4\sqrt{v})$. As defined in \eqref{eq: theta ubar}, $\underline{\theta}_{\min}(\alpha;c)\geq \varphi(\alpha;c)$. Since $\ubar{\theta}_{\min}(\alpha_l;c)$ is endpoint-dominated, then $\ubar{\theta}_{\min}(\alpha_d;c)<\varphi(\alpha_l;c)$ such that $\varphi(\alpha_d;c)<\varphi(\alpha_l;c)$ which implies $\alpha_d(c)<\alpha_{l}(c)$, i.e., $s(c)<1$. Due to $\sigma_{\alpha_d}(\epsilon,c)>\sigma$ but $\alpha_d(c)<\sigma c$, then $\alpha_d$ must lie strictly above its chord somewhere on
$(\epsilon,c)$ by Lemma.\ref{lem:necessary-concave}; hence it is strictly concave on a nontrivial subset. \emph{(ii)} 
If $\alpha_d$ is linear with slope $\tilde\sigma>\sigma$, then
$\alpha_d(c)=\tilde\sigma c>\sigma c$ such that $\ubar{\theta}_{\min}(\alpha_d;c)>\ubar{\theta}_{\min}(\alpha_d;c)$ 
If $\alpha_d$ is strictly convex and $\sigma_{\alpha_d}(\epsilon,c)>\sigma$, then 
$\alpha_d(c)>\sigma c$ by Proposition. \ref{prop:ordering}, such that the actuation increases.
\end{proof}

Consequently, strict concavity is necessary to simultaneously lower $\theta_{\min}$  and increase the guaranteed nominal decay rate over the linear baseline. 

\subsection{Feasibility–Preserving Acceleration}\label{subsec:accel-feasible}
In this subsection, considering a fixed $\theta$, we show that if a linear baseline is feasible on $\mathbb{X}$, then any strict
margin on a sublevel set can be converted into a \emph{concave} comparison
function that (i) remains feasible on
$\mathbb{X}$, and (ii) strictly improves the nominal exponential rate on
the evaluation interval $[\epsilon,c]$ contained in that sublevel set.

\begin{thm}[Feasibility–preserving acceleration]\label{thm:FPA}
Let $V$ be a smooth CLF on $\mathbb{X}$. Suppose the linear baseline, $\forall x\in\mathbb{X}$,
\begin{subequations}
\begin{equation}\label{eq:baseline}
  \exists\,u\in\mathbb{U}(\theta),~  
  L_fV(x)+L_gV(x)u \le -\sigma V(x),
\end{equation}
holds for some $\sigma>0$. Assume there exist $c_0\in(0,c_{\max}]$ and
$\sigma_0>0$ such that for all $x\in\Omega(V,c_0)$,
\begin{equation}\label{eq:margin}
  \exists\,u\in\mathbb{U}(\theta),~ 
  L_fV(x)+L_gV(x)u \le -(\sigma+\sigma_0) V(x)
\end{equation}
Then there exists $\alpha_{cave}\in\mathcal{K}_{cave}$ such that $\forall x\in\mathbb{X}$,
\begin{align}
  \exists\,u\in\mathbb{U}(\theta),~ 
  L_fV(x)+L_gV(x)u \leq -\alpha_{cave}(V(x)),\label{eq:cave-feasible}
\end{align}
and for every $c\in(0,c_{\max}]$ there exists $\epsilon_0\in(0,\min\{c,c_0\})$
such that for all $\epsilon\in(0,\epsilon_0)$,
\begin{equation}\label{eq:rate_improve}
  \sigma_{\alpha_{cave}}(\epsilon,c)\ >\ \sigma.
\end{equation}
\end{subequations}
\end{thm}

\begin{proof}
(i) Feasibility: Pick a \emph{concave factor} $s:[0,\infty)\!\to\!(0,\infty)$,
strictly decreasing, and satisfies
\begin{align*}
    s(0)=(1+\frac{\sigma_0}{\sigma}), ~s(c_0)=1
\end{align*}
Define $\alpha_{cave}(v)\doteq s(v)\,\sigma v$. Then $\alpha_{cave}\in\mathcal{K}_{cave}$,
$\alpha_{cave}(v)\le (\sigma+\sigma_0)v$ for $v\in(0,c_0]$ and $\alpha_{cave}(v)\le\sigma v$ $v\in(c_0,c_{\max}]$.
Thus \eqref{eq:cave-feasible} holds. (ii) Acceleration: If $c<c_0$, $s(v)>1$ holds for $v\in[\epsilon,c]$ with any $\epsilon<c$. Thus, in this case $\sigma_{\alpha_{cave}}(\epsilon,c)>\sigma$ for any $[\epsilon,c]\subset(0,c_0]$. If $c>c_0$, we have $\sigma_{cave}(\epsilon,c_0)>\sigma$ and $\sigma_{cave}(c_0,c)<\sigma$. By Lemma 5, we obtain there always exists a sufficiently small $\epsilon_0\in(0,c]$ such that for all $\epsilon\in(0,\epsilon_0)$, $\sigma_{\alpha_{cave}}(\epsilon,c)\in(\sigma,\sigma+\sigma_0)$.
\end{proof}
\begin{corol}[Uniform improvement on an interval]
Let $c\in(0,c_0]$ and  $s:\mathbb{R}_{\geq0}\to\mathbb{R}_{>0}$ be a  concave factor normalized by $s(c)=1$ and $s(0)=1+\sigma_0/\sigma$ (hence $s(v)>1$ for all $v\in(0,c)$). Define $\alpha_{cave}(v)=s(v)\sigma v$ and let $\alpha_l(v)=\sigma v$. Then, for every $\epsilon\in(0,c)$, $\sigma_{\alpha_{cave}}(\epsilon,c)>\sigma$.
\end{corol}

\subsection{Example on the saturated single integrator}
Consider $\dot x=u$ with $\|u\|_\infty\le \theta$ and the CLF $V(x)=x^2$. Since $L_fV\equiv0$ and $L_gV=2x$, by definition
\[
  D_{\max}(x,\theta)=\bar\alpha(\theta,V)=2\theta\sqrt{V}.
\]
Then the (best) global exponential rate $ \sigma^\star={2\theta}/{\sqrt{c}}$ for $0<V\leq c$.
 The required 
actuation level for a  linear comparison $\alpha_l(V)=\sigma V$ is $\theta_{\min}(\alpha_{l};c)=\ubar{\theta}_{\min}(\alpha_l,c)=\sigma\sqrt{c}/2$. With a bang–bang law $u=-\theta\,\mathrm{sgn}(x)$,
\[
   \dot V=2xu=-2\theta|x|=-2\theta\sqrt{V}=-\bar\alpha(\theta,V),
\]
so the windowed nominal rate on $[\epsilon,c]$ is
\[
  \sigma_{\bar\alpha}(\epsilon,c)
  \;=\; \frac{\ln(c/\epsilon)}{\displaystyle\int_{\epsilon}^c \frac{dV}{2\theta\sqrt V}}
  \;=\; \frac{2\theta\,\ln(\sqrt c/\sqrt\epsilon)}{\sqrt c-\sqrt\epsilon}.
\]

\subsubsection*{Numerical preview $[\epsilon,c]=[10^{-4},100]$}
Here $\sigma^\star=0.2\,\theta$ while
$\sigma_{\bar\alpha}(\epsilon,c)\approx 1.382\,\theta$, i.e., about $6.9$ times
faster than the best linear rate with the same $\theta$. The endpoint–relaxation ratio
\[
  r_{\bar\alpha}(\epsilon,c)
  \;=\; \frac{\bar\alpha(\theta,c)}{\sigma_{\bar\alpha}(\epsilon,c)\,c}
  \;\approx\; 0.145,
\]
quantifies the reduction in endpoint value and, on this window, matches the required–actuation ratio at equal nominal rate (about $85.5\%$ savings).

\subsubsection*{On trade-offs of non-Lipschitz concave comparison}
The bang–bang law (similar to a first–order sliding controller\cite{LevantTACSMC}) is discontinuous at
$x=0$ and prone to chattering. In sampled implementations, an
insufficient sampling period can aggravate this: even when the comparison
function is Lipschitz in $V$, too large an effective slope can excite
high–frequency switching. To retain the decay advantages of concavity while avoiding chattering, the next section introduces a
\emph{Lipschitz–aware concave factor} $s(\cdot)$ shaping $\alpha(V)=s(V)\sigma V$
with an explicit slope bound,
chosen to match the sampling period and actuator bandwidth. With the rational
factor in \eqref{eq:rational-factor}, one has
\[
  0\ \le\ s_{\mathrm{rat}}(V)+V s'_{\mathrm{rat}}(V) \le k_{\max}
  \Longrightarrow
  \Big|\tfrac{d}{dV}\alpha(V)\Big|\le \sigma k_{\max},
\]
so $k_{\max}$ directly limits the effective slope.

\section{Constructive Concave Factors}\label{sec:construct}
The design variable in Theorem~\ref{thm:FPA} is the \emph{concave factor} $s(\cdot)$. In practice (simulation or sampled control), overly fast decay laws can cause chattering or numerical oscillations when the sampling period is not small enough \cite{numeicalchattering}. To address this, we use \emph{Lipschitz-aware} concave factors whose slope can be explicitly capped. A convenient  $C^\infty$ family on $(0,\infty)$ is  \emph{rational} form
\begin{equation}\label{eq:rational-factor}
  s_{\mathrm{rat}}(v)=\frac{k_{\min} v + k_{\max}\ell}{v + \ell},
  \quad 0\leq k_{\min}<k_{\max}, ~\ell>0,
\end{equation}
which satisfies $s_{\mathrm{rat}}(0)=k_{\max},s_{\mathrm{rat}}(\infty)=k_{\min}$, and
\begin{align*}
    s'_{\mathrm{rat}}(v)=-\frac{(k_{\max}-k_{\min})\ell}{(v+l)^2},~
     s''_{\mathrm{rat}}(v)=\frac{2(k_{\max}-k_{\min})\ell}{(v+l)^3}
\end{align*}
We can design a concave comparison function by multiplying a linear comparison with a concave factor.
\begin{lem}[Concavification via a rational factor]
\label{lem:rat-concave}
For any $\alpha_l(v)=\sigma v\in\mathcal{K}_l$ with $\sigma>0$, the product
\[
  \alpha(v)\;\doteq\; s_{\mathrm{rat}}(v)\,\alpha_l(v)=s_{\mathrm{rat}}(v)\sigma v
\]
belongs to $\mathcal{K}_{cave}$ and is locally Lipschitz on $[0,\infty)$.
\end{lem}

\begin{proof}
Let $F(v)\doteq v\,s_{\mathrm{rat}}(v)=\frac{k_{\min}v^2+k_{\max}\ell v}{v+\ell}$. A direct differentiation yields $F'(v)>0$ and $F''(v)<0$ impling concavity. Multiplying by the positive constant $\sigma$ preserves strict concavity. Local Lipschitzness follows from $|F(v_1)-F(v_2)|\leq k_{\max}|v_1-v_2|$.
\end{proof}

\begin{corol}[Composition of rational factors]
\label{cor:power-comp}
For $p\in(0,1)$, the composed factor $\tilde s(v)\doteq s_{\mathrm{rat}}(v^{p})$ is again a concave factor in the sense that
$\tilde F(v)\doteq \tilde{s}(v)v=s_{rat}(v^p)v$ 
is strictly concave on $(0,\infty)$.
\end{corol}
\begin{proof} \emph{Hint:} By differentiation, we have $\tilde{F}'(v)>0,\tilde{F}''(v)<0$.
\end{proof}

This shows that composing the rational factor with a sublinear power $v\mapsto v^p$ ($p\in(0,1)$) preserves the “concave factor” property (i.e., $v\mapsto s(v^p)v$ stays strictly concave). It’s a useful strengthening of the rational family since it adds an extra shape parameter $p$ to steepen the near-origin gain while relaxing upper endpoint tuning.

\subsection{Concave–shape tuning}
We start from a linear baseline $\alpha_l(v)=\sigma v$, where $\sigma$ is the best
(global) exponential rate available on $\mathbb{X}$. Consider an evaluation window
$[\epsilon,c]$ and design a rational concave factor
\[
  s_{\rm rat}(v) = \frac{k_{\min} v + k_{\max}\,\ell}{v+\ell},
  \quad 0<k_{\min}<1<k_{\max},\ \ell>0,
\]
so that the concave comparison $\alpha(v)\doteq s_{\rm rat}(v)\sigma v$ 
meets performance and feasibility goals on $[\epsilon,c]$. The target nominal rate $\sigma^\star>\sigma$ such that $\sigma_{\alpha}(\epsilon,c)\geq\sigma^\star$.

\subsubsection{Endpoint normalization}
We set $c=V(x_0)$ and impose $s_{rat}(c)=r\leq1$ so the endpoint  matches the baseline:
\begin{equation}\label{eq:s-at-c}
  s_{rat}(c)=r 
  \iff \ell = \frac{(r-k_{\min})c}{k_{\max}-r}.
\end{equation}
This leaves two free shape parameters $(k_{\min},k_{\max})$ with $k_{\min}\in(0,1)$, $k_{\max}>0$. 

\subsubsection{Exact nominal rate under $s_{rat}$}
The guaranteed crossing time by $\alpha$ is
\begin{align*}
   T_\alpha(\epsilon,c)&=\frac{1}{\sigma}\int_{\epsilon}^{c}\frac{y+\ell}{y(k_{\min}y+k_{\max}\ell)}\,dy \\
   &=\frac{1}{\sigma}\bigg[
      \frac{1}{k_{\max}}\ln\!\frac{c}{\epsilon}
      +\frac{k_{\max}-k_{\min}}{k_{\max}k_{\min}}
        \ln\!\frac{k_{\min}c+k_{\max}\ell}{k_{\min}\epsilon+k_{\max}\ell}
    \bigg]
\end{align*}
Hence, the closed form  rate (cf.\ \eqref{eq:sigma-alpha}) is  
\begin{equation}\label{eq:sigma-closedform}
  \sigma_{\alpha}(\epsilon,c)
  = 
    \frac{\sigma\,k_{\max}\ln(c/\epsilon)}{\displaystyle
      \ln(c/\epsilon)
      + \frac{k_{\max}-k_{\min}}{k_{\min}}\,
        \ln\!\frac{k_{\min}c+k_{\max}\ell}{k_{\min}\epsilon+k_{\max}\ell}}>\sigma^\star.
\end{equation}
By \eqref{eq:sigma-closedform}, we can select a  low $k_{\min}<1$ and solve $k_{\max}$, or select a sufficiently high $k_{\max}>\sigma^\star/\sigma$ and solve $k_{\min}$.
\subsubsection{Feasibility check (necessity)}
The sufficient condition of feasibility  is that $\alpha(V(x))\leq D_{\max}(\theta,x)$ along the trajectory $x(t)$. However, the trajectory $x(t)$ is usually not available ahead.  The level-wise cap 
bound presented in \eqref{eq:level-wise condition}
\begin{align*}
    \alpha(v)\le \bar\alpha(\theta,v)\leq k_3 v + k_4\,\theta \sqrt{v},\qquad \forall v\in(0,c_{\max}].
\end{align*}
 provides a necessary condition, offering a good estimate. 
Thus,  choose $s_{\mathrm{rat}}(\epsilon)<\frac{k_3}{\sigma}+\frac{k_3\theta}{\sqrt{\epsilon}}$. Combined with \eqref{eq:s-at-c}, this yields an explicit admissible region for $(k_{\min},k_{\max})$ with $s(c)=r$.

\subsubsection*{Practical tuning recipe}
\begin{enumerate}
  \item Choose  $[\epsilon,c]$ and target nominal rate $\sigma^\star$ and baseline $\sigma$.
  \item Pick a sufficiently low $k_{\min}\in[0,1)$ for keeping feasibility. Pick an aggressive $k_{\max}>\sigma^\star/\sigma>1$.
  \item Enforce $s_{\mathrm{rat}}(c)=r\leq1$ via \eqref{eq:s-at-c} to get $\ell$.
  \item Verify $\sigma_{\alpha}(\epsilon,c)\geq\sigma^\star$. If not, increase $k_{\max},r$ and decrease $k_{\min}$. 
  \item Verify feasibility using the caps $D_{\max}(x,\theta), \bar{\alpha}(\theta,v)$ if available. If not, directly tune based on simulation.
  \item If feasibility fails (input is saturated), decrease $r$ for relaxing constraint in high-level set and  $k_{\max}$ in low-level set.
  \item If needed, according to the result of Corollary \ref{cor:power-comp}, use $s_{rat}(v^p)$, e.g., $p=0.5$.
\end{enumerate}

\section{CLF--QP with Concave Comparison}\label{sec:CLFQP}
We integrate the concave comparison framework into a standard CLF--QP pipeline in which the Lyapunov decay rate is designed directly. 
The key idea is to replace the linear comparison $\dot V(x,u)\le -\sigma V$ with a \emph{concave} constraint
\begin{equation*}
  \dot V(x,u)\le -s\big(V(x)\big)\sigma V(x)
  \label{eq:concave-decay-constraint}
\end{equation*}
where $s:\mathbb{R}_{\geq0}\to\mathbb{R}{>0}$ is a rational concave factor (unless noted, $s(\cdot)$ is the rational concave factor in following contexts), so as to preserve feasibility under actuator limits while strictly increasing the windowed nominal rate $\sigma_{\alpha}(\epsilon,c)$ on a target window $[\epsilon,c]$ (Sec.~\ref{sec:concave}). \label{sec:clfqp-main}

\subsection{Concave-Shaped CLF–QP}\label{subsec:clfqp-formulations}
Consider the control–affine system \eqref{eq:system} with a smooth CLF $V$ and an input bound $\|u\|_\infty\le\theta$.
Given a baseline exponential rate $\sigma>0$, define
\[
  \alpha_c(V) \;\doteq\; s(V)\,\sigma\,V \;\in\; \mathcal{K}_{\mathrm{cave}}.
\]
The function $\alpha_c$ reshapes the linear baseline while reducing the endpoint value if we impose $s(c)<1$ at the endpoint $V=c$ in the estimate window (endpoint normalization).

\subsubsection*{Hard constraint (feasibility required)}
The CLF–QP with a hard concave comparison constraint is
\begin{align}
\min_{u}\quad & u^\top 
u \nonumber\\
\text{s.t.}\quad
& L_fV(x) + L_gV(x)u + s\big(V(x)\big)\sigma V(x) \le 0,\label{eq:clfqphard}\\
& \|u\|_\infty \;\le\; \theta \nonumber
\end{align}
When $\theta=\infty$, \eqref{eq:clfqphard} reduces to the \emph{mini-norm} CLF controller (see \cite{ames2013towards,ames2014rapidly} for the details).
When \eqref{eq:clfqphard} holds for all $x\in\mathbb{X}$, it yields
$\dot V \le -\alpha_c(V)=-s(V)\sigma V$ globally. On any window $[\epsilon,c]\subset(0,c_{\max}]$, if $s(c)=1$, implying $s(v)>1$ on a subset of $(0,c)$ with positive measure, then
\[
  T_{\alpha_c}(\epsilon,c) \;<\; T_{\alpha_l}(\epsilon,c)\qquad\text{and}\qquad
  \sigma_{\alpha_c}(\epsilon,c) \;>\; \sigma,
\]
i.e., the guaranteed decay strictly improves while keeping the same endpoint. If $s(c)=r<1$, than we have to tune $k_{\min},k_{\max}$ to achieve the target using \eqref{eq:sigma-closedform}.

\subsubsection*{Soft constraint (always feasible)}
Introducing a nonnegative slack $\delta$ yields the soft CLF–QP:
\begin{align}
\min_{u,\delta}\quad & \,u^\top H u \;+\; q\delta^2 \label{eq:clfqpslack}\\
\text{s.t.}\quad
& L_fV(x) + L_gV(x)u + s\big(V(x)\big)\sigma V(x) \le \delta,\nonumber\\
& \|u\|_\infty \le \theta,\quad \delta\ge0,\nonumber
\end{align}
with $H\succ0$ and $q>0$. This is always feasible; the optimal slack $\delta^\star(x)$ induces the closed–loop inequality
\begin{equation}\label{eq:Vdot-soft}
  \dot V(x) \le -s\big(V(x)\big)\sigma V(x) + \delta^\star(x).
\end{equation}

\begin{prop}[Acceleration for soft CLF--QP]
\label{prop:simple-soft}
Let $\alpha_{l}(V)=\sigma V$ and $\alpha_{c}(V)=s(V)\,\sigma V$, with normalized rational factor $s(c)=1$ (hence $s(V)>1$ on $(\epsilon,c)$).
Consider the soft CLF--QP with slack $\delta\ge 0$ as in \eqref{eq:clfqpslack}, and assume feasibility on $[\varepsilon,c]$.
Then the following each imply strict improvement:
\begin{enumerate}
\item  If $\delta^\star(x)=0$ for all states with $V(x)\in[\epsilon,c]$, then
$T_{\alpha_c}(\epsilon,c)<T_{\alpha_l}(\epsilon,c)$, equivalently $\sigma_{\alpha_c}(\epsilon,c)>\sigma$.
\item  If there exists an interval $(a,b)\subset(\epsilon,c)$ such that $\delta^\star(x)=0$ whenever $V(x)\in(a,b)$, then the same strict inequality holds.
\item  If there exists an interval $(a,b)\subset(\epsilon,c)$ such that
$\alpha_c(V)-\delta^\star(x)>\alpha_l(V)$ whenever $V(x)\in(a,b)$, then the same strict inequality holds.
\end{enumerate}
\end{prop}
\begin{proof}
    On any interval where the enforced bound satisfies $-\alpha_c(V)+\delta^\star<-\alpha_l(V)$, one has
$\int_{\epsilon}^{c}\frac{dv}{\alpha_c(v)-\delta^\star}<\int_{\epsilon}^{c}\frac{dv}{\alpha_l(v)}$,
yielding $T_{\alpha_c}(\epsilon,c)<T_{\alpha_l}(\epsilon,c)$ and hence $\sigma_{\alpha_c}(\epsilon,c)>\sigma$.
\end{proof}
\subsubsection*{Numerical issues induced by slack}
With a soft CLF constraint, the optimizer returns a slack $\delta^\star(x)\ge 0$ that is
rarely \emph{exactly} zero in floating point even when a minimum-norm controller is feasible.
The certified instantaneous rate is lower-bounded by
\[
-\frac{\dot V}{V}\ \ge\ s(V)\sigma\;-\;\frac{\delta^\star}{V}.
\]
When $V$ is moderate, a small absolute error $\Delta$ in the slack (solver tolerance,
round-off) has negligible effect: $(\delta^\star+\Delta)/V\approx \delta^\star/V$. 
However, as $V\!\to\!0$, division by $V$ amplifies any residual: even if $\delta^\star=0$,
a tolerance-sized perturbation $\Delta$ yields $(\delta^\star+\Delta)/V=\Delta/V\gg 0$
(e.g., $\Delta=10^{-4}$ and $V=10^{-4}$ produce a unit drop in the lower bound). 
This makes the soft CLF–QP appear significantly more conservative than the minimum-norm
(hard) controller near the origin, purely for numerical reasons.

\subsection{Existing CLF--QP frameworks for accelerating decay}
Prior CLF--QP approaches accelerate decay primarily by \emph{modulating a global exponential rate} rather than shaping the comparison function. Two representative lines are: (i) hybrid/time–scale scheduling via rapidly exponentially stabilizing CLFs (RES--CLFs) \cite{ames2014rapidly}; and (ii) optimization of the rate as a decision variable (''flexible'' CLF--QP) \cite{FlexCLF} and \cite{OD-CLF-QP}.

\subsubsection{Rapidly exponentially stabilizing CLF--QP}
RES--CLF details appear in \cite{ames2014rapidly}. The key is a family $V_{\varepsilon}$ with
\begin{align}
    \begin{split}
       &c_1\|x\|^2 \le V_\varepsilon(x) \le \tfrac{c_2}{\varepsilon^2}\|x\|^2 \\
        \inf_{u\in\mathbb{R}^m}&\Big\{ L_f V_\varepsilon(x) + L_g V_\varepsilon(x)\,u + \tfrac{c_3}{\varepsilon} V_\varepsilon(x) \Big\} \le 0
    \end{split}
\end{align}
where $c_1,c_2,c_3>0$ and $\varepsilon\in(0,1)$. Feasibility yields the decay guarantee
\[
  \dot V_\varepsilon(x) \le -\tfrac{c_3}{\varepsilon}\, V_\varepsilon(x) \doteq -\alpha_\varepsilon\!\big(V_\varepsilon(x)\big),
\]
so decreasing $\varepsilon$ increases the global exponential rate $c_3/\varepsilon$. Under input constraints $\|u\|_\infty\le \theta$, however, the maximum exponential rate is capped pointwise by
\[
  \sigma_{\max}(V) \doteq {\bar\alpha(\theta,V)}/{V},
\]
which typically decreases with $V$ (e.g., for $\bar\alpha(\theta,V)\leq k_3 V + k_4 \theta \sqrt{V}$). Consequently, feasibility at high levels often forces evaluation with a \emph{larger} $\varepsilon$ (hence a weaker windowed rate). Any practical acceleration then appears on sublevel sets where the soft constraint is inactive ($\delta^\star=0$).  This is consistent with the concave perspective: acceleration arises by reallocating decay away from the endpoint and toward regions with available margin.

\subsubsection{Flexible CLF--QP}
A flexible CLF--QP treats the rate as a decision variable:
\begin{align}
  \min_{u,\sigma}\quad & (1-\kappa(x))\,u^\top u \;+\; \kappa(x)\,\big(\sigma_{\max}-\sigma\big)^2 \label{eq:flex-obj}\\
  \text{s.t.}\quad & L_f V(x) + L_g V(x)\,u + \sigma\,V(x) \;\le\; 0, \label{eq:flex-clf}\\
  & \sigma_{\min} \le \sigma \le \sigma_{\max}, \qquad \|u\|_\infty \le \theta, \label{eq:flex-bounds}
\end{align}
with a locally Lipschitz weight $\kappa:\mathbb{R}^n\!\to[\kappa_{\min},\kappa_{\max}]$, $0<\kappa_{\min}\le \kappa_{\max}<1$. The optimizer selects $\sigma^\star(x)\in[\sigma_{\min},\sigma_{\max}]$; as $\kappa(x)\to 0$ the problem emphasizes input economy ($\sigma^\star\!\to\sigma_{\min}$), and as $\kappa(x)\to 1$ it favors larger rates ($\sigma^\star\!\to\sigma_{\max}$). With input bounds, the achievable rate is still limited by the endpoint cap via $\sigma^\star(x)\le \sigma_{\max}\!\big(V(x)\big)$, which usually decreases with $V$; hence global acceleration remains constrained by high-level feasibility.

\subsubsection{Discussion}
Regardless of the CLF--QP variant, under the standing assumptions and for a sufficiently large window $V(x)\in[\epsilon,c]$, actuation demand concentrates at higher levels where $\sigma_{\max}(V)$ is smallest. If we forgo enforcing a \emph{uniform} pointwise rate and instead certify the \emph{windowed nominal rate} $\sigma_{\alpha}(\epsilon,c)$, then shaping $\alpha$ becomes a powerful degree of freedom. The proposed concave factor $s(V)$ (with controlled slope) dynamically exploits any available slack, preserves feasibility under the same endpoint cap, and provides a computable $\sigma_{\alpha}(\epsilon,c)$ for offline tuning. If the soft CLF--QP returns $\delta^\star=0$, the target nominal rate is achieved; otherwise, the violation certifies that the target is unattainable under the input constraint.
\section{Case Studies}\label{sec:casestudy}
We present two case studies. The first (an inverted pendulum with torque saturation) is a
\emph{toy} yet representative example that walks through the main concepts:
(i) actuator–induced caps on Lyapunov decay; (ii) the endpoint–limited baseline exponential
rate; and (iii) acceleration via concave comparisons in a CLF–QP. The second (attitude control
for a quadrotor) illustrates that the same design carries over to a practical
system subject to tight actuator limits.

\subsection{Common evaluation protocol}
For both cases we use a smooth control Lyapunov function (CLF) $V$ and enforce a soft CLF–QP
with input bounds:
\begin{align}
\begin{split}
    \min_{u,\;\delta\ge0}\quad & u^\top H u + q\,\delta^2 \\
\text{s.t.}\quad
& L_fV(x) + L_gV(x)\,u + \alpha\!\big(V(x)\big)\;\le\; \delta,\\
& \|u\|_\infty\le \theta.
\end{split}
\end{align}
We compare $\alpha\in\{\alpha_l,\alpha_c\}$ where the \emph{linear} baseline is
$\alpha_l(V)=\sigma V$ and the \emph{concave} comparison is
$\alpha_c(V)=s_{\mathrm{rat}}(V)\,\sigma V$ with the rational concave factor
\[
s_{\mathrm{rat}}(v)\;=\;\frac{k_{\min}(v/c)+k_{\max}\,\ell}{(v/c)+\ell},\qquad
0<k_{\min}<1<k_{\max}.
\]
It is normalized at the window endpoint $c$ via $s_{\mathrm{rat}}(c)=r\le 1$, hence
$\ell=(r-k_{\min})/(k_{\max}-r)$ (see~\eqref{eq:s-at-c}). The choice $r=1$ enforces the
\emph{same endpoint} as the linear baseline ($\alpha_c(c)=\alpha_l(c)=\sigma c$); choosing
$r<1$ intentionally relaxes the endpoint to reduce peak actuation.
Given an initial condition $x_0$, we set $c=V(x_0)$ and $\epsilon=\xi c$ with
$\xi\in\{10^{-4},10^{-3},10^{-2}\}$. The baseline rate $\sigma$ is taken from the cited
reference for each system.

\subsubsection*{Metrics}
We mainly report: (i) the normalized decay $ V(x(t))/V(x_0)$; (ii) the exact crossing time $T(\epsilon,c)$ and the induced nominal rate
$\sigma_{\mathrm{nom}}(\epsilon,c)\doteq \ln(c/\epsilon)/T(\epsilon,c)$; (iii) the instantaneous exponential rate
$-\dot V/V$; and (vi) actuation usage $\|u\|_\infty$ and energy
$\int_0^{T(\epsilon,c)} u^\top u\,dt$.

\subsection{Inverted pendulum}
\label{subsec:pendulum-clfqp}
We consider a torque–saturated pendulum with viscous friction (as in \cite{choi2020cbfclfhelper}):
\begin{equation}\label{eq:ip-dyn}
  \dot\psi = \omega,\quad
  \dot\omega = \frac{m g l}{I}\sin\psi - \frac{b}{I}\,\omega - \frac{1}{I}u,
  \quad \|u\|_\infty\le \theta,
\end{equation}
with $m,l,I,b>0$, gravity $g$, state $x=[\psi,~\omega]^\top$, and torque bound $\theta$. 
We adopt a quadratic CLF $V(x)=x^\top P x$ with $P=P^\top\!\succ0$ from
\begin{equation}\label{eq:ip-lyap}
  A_{\rm clf}^\top P + P A_{\rm clf} = -\sigma_{\rm clf}\,Q_{\rm clf},\qquad
  Q_{\rm clf}\succ0,\ \ \sigma_{\rm clf}>0,
\end{equation}
where
\[
  A_{\rm clf} = \begin{bmatrix}
    0 & 1\\[2pt]
    \frac{m g l}{2 I} - K_1 & -\frac{b}{I}-K_2
  \end{bmatrix},\qquad K_1,K_2>0,
\]
which encodes a PD template and yields a valid CLF for \eqref{eq:ip-dyn} on a large region around $\psi=0$ \cite{choi2020cbfclfhelper}.

\subsubsection*{Parameter settings}
Given $x_0=[\pi/4\ \ 0.05]^\top$, set $c=V(x_0)$ and $\epsilon=\xi\,c$ with $\xi\in\{10^{-4},10^{-3},10^{-2}\}$.
Sampling time $\mathrm dt=1$\,ms. Parameters as in \cite{choi2020cbfclfhelper}:
$(m,l,b,g)=(1,1,0.01,9.81)$, $I=ml^2/3$, $K_1=6$, $K_2=5$, $\sigma_{\rm clf}=3$,
$Q_{\rm clf}=\mathrm{diag}(1,1)$, $q=10^5$, $\theta=10$, baseline $\sigma=3$
(near the best feasible exponential rate for $\theta=10$).
We set $k_{\max}=2.3$, $k_{\min}=0.1$, and $r\in\{1,0.9,0.8,0.7,0.6\}$.
All QPs use \texttt{quadprog}; dynamics are integrated by \texttt{ode45} with ZOH over each sampling interval.

\subsubsection*{Results}
To expose endpoint domination, we plot the decay rate cap along the trajectory,
$D_{\max}(x,\theta)$, overlaid with the candidate comparison functions (Fig.~\ref{fig:ip_dmax}):
the endpoint value $\alpha(c)$ pins the assignable exponential rate and strictly limits what can be certified. 
We then simulate the soft CLF--QP for the linear baseline and for concave designs with
$s_{\mathrm{rat}}(c)=r$. As shown in Table~\ref{tab:ip_metrics}, relaxing the endpoint ($r<1$)
reduces the peak actuation $\|u\|_\infty$, confirming the significance of the endpoint cap.

As predicted, with \emph{endpoint normalization} ($r=1$), the concave design yields a strictly larger
windowed nominal rate $\sigma_{\mathrm{nom}}(\epsilon,c)$ (equivalently, shorter $T(\epsilon,c)$) than the linear baseline. 
For $r\in\{1,0.9,0.8,0.7,0.6\}$ the controller maintains a large nominal rate while lowering peak actuation
(Fig.~\ref{fig:inverted_result}); notably, this acceleration is obtained \emph{without} increasing energy
$\int_0^{T_\epsilon} u^\top u\,\mathrm dt$.

Finally, with the mini-norm (hard) CLF controller ($\delta\equiv0$), the instantaneous rate
tracks the designed profile $s(V)\sigma$ and remains monotone increasing toward equilibrium, 
whereas the soft CLF--QP can exhibit small dips at very low $V$ due to the amplification of $\delta/V$
(Compare Fig.~\ref{fig:inv-mininorm} with  Fig.~\ref{fig:inverted_result}).

\begin{figure}
    \centering
    \includegraphics[width=1\linewidth]{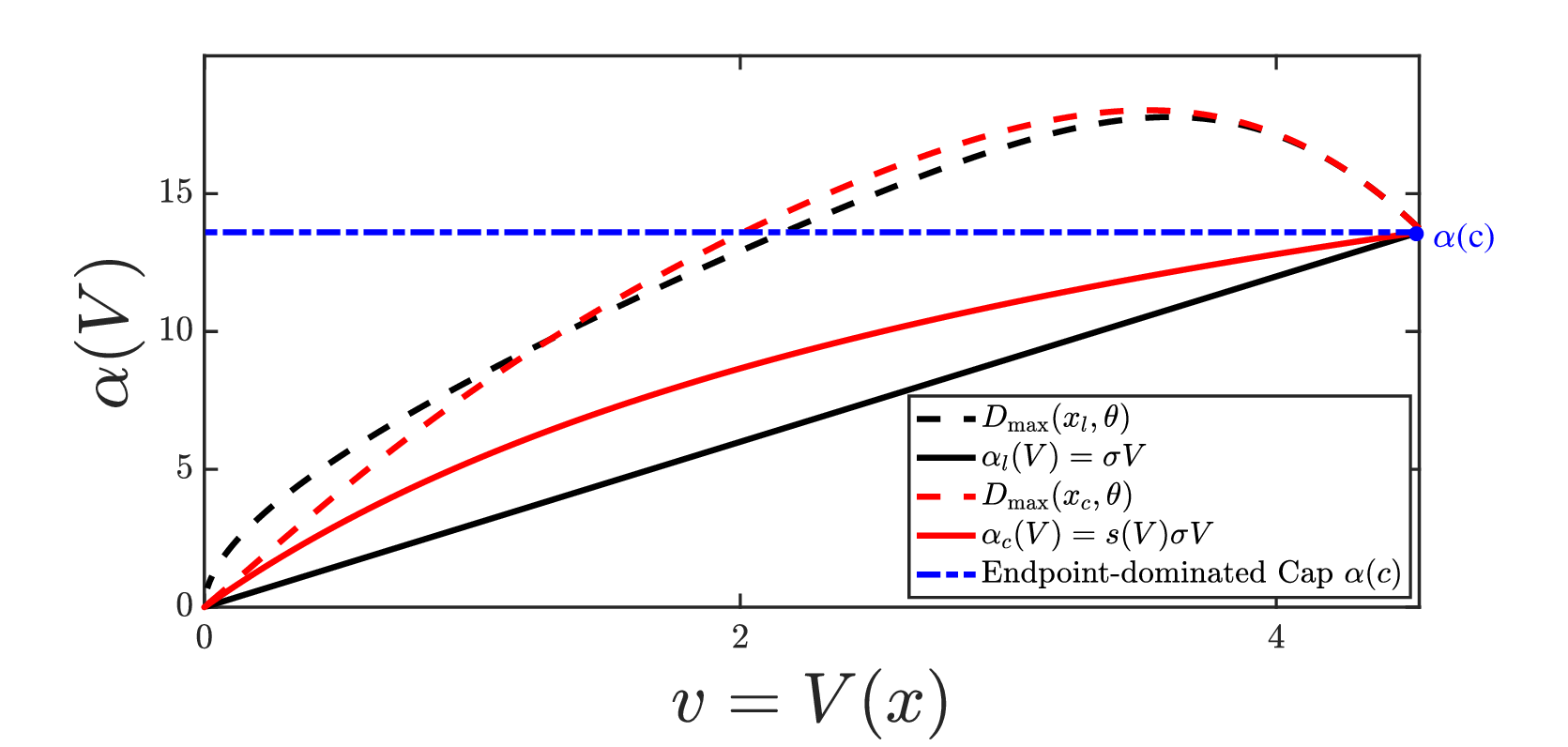}
    \caption{Assignable decay-rate cap $D_{\max}(x,\theta)$ along the trajectory for $\theta=10$ and the comparison functions. The endpoint value $\alpha(c)$ strictly restricts the assignable exponential rate.}
    \label{fig:ip_dmax}
\end{figure}

\begin{figure}
    \centering
    \includegraphics[width=\linewidth]{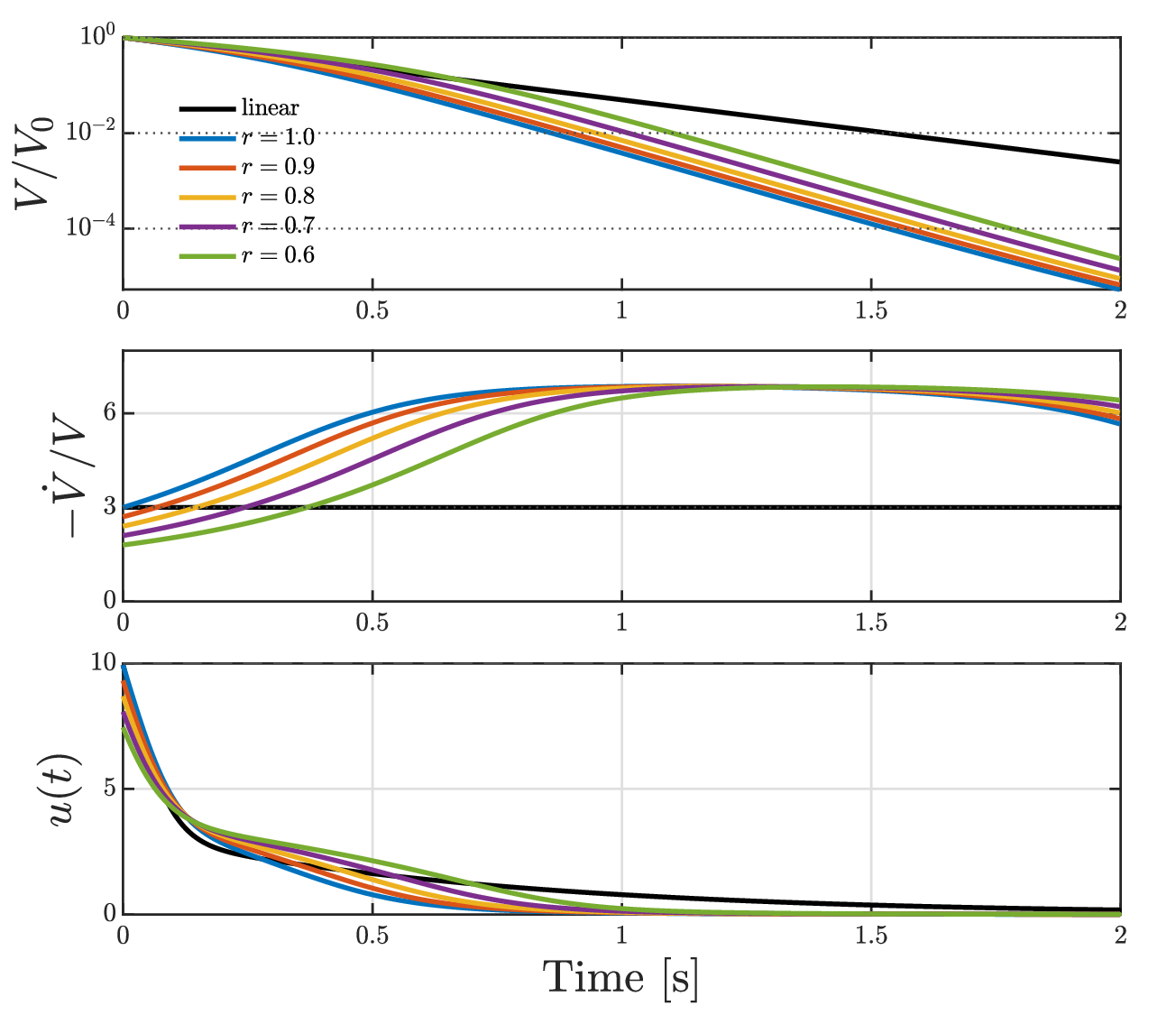}
    \caption{Inverted pendulum: Normalized Lyapunov response $V(x(t))/V(x_0)$, instantaneous exponential rate $-\dot V/V$, and control input $u$.
    For windows with $\epsilon\le 10^{-1}c$, all concave designs achieve a strictly larger nominal rate than the linear baseline (shorter crossing times). For very small $V\lesssim 10^{-4}c$, the instantaneous rate dip because $\delta/V$ grows as $V\to0$.
    In contrast, with the hard (mini-norm) CLF controller ($\delta\equiv0$), the instantaneous rate remains  increasing toward equilibrium, consistent with the slack analysis as shown in Fig.\ref{fig:inv-mininorm}.}
    \label{fig:inverted_result}
\end{figure}
\begin{figure}
    \centering
    \includegraphics[width=\linewidth]{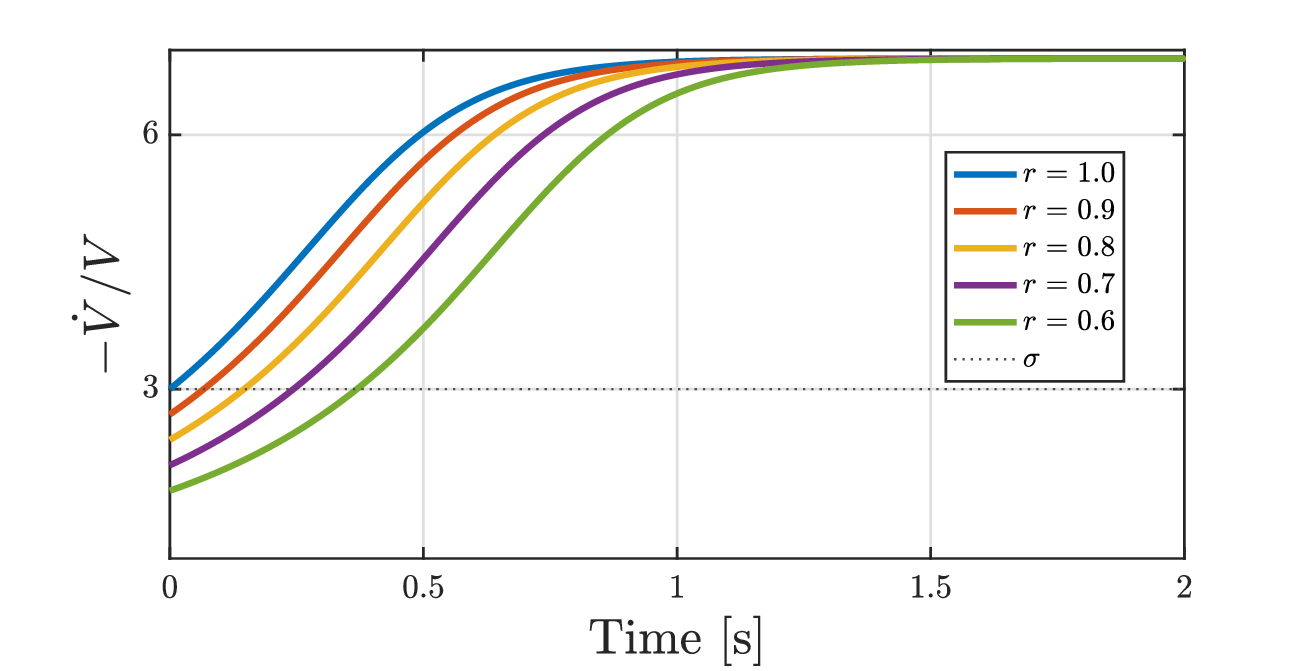}
    \caption{Inverted pendulum: Mini-norm controllers with concave constraints where the instantaneous rate tracks the designed profiles.}
    \label{fig:inv-mininorm}
\end{figure}

\begin{table*}[t]
\centering
\caption{Inverted Pendulum: Performance Metrics}
\label{tab:ip_metrics}
\begin{tabular}{l ccc ccc ccc c}
\toprule
& \multicolumn{3}{c}{$\epsilon_1 = 10^{-2}c$}
& \multicolumn{3}{c}{$\epsilon_2 = 10^{-3}c$}
& \multicolumn{3}{c}{$\epsilon_3 = 10^{-4}c$}
& \multirow{2}{*}{$\|u\|_\infty$} \\
\cmidrule(lr){2-4} \cmidrule(lr){5-7} \cmidrule(lr){8-10}
Controller 
& $T_{\epsilon_1}$ & $\sigma_{\text{nom}}$ & $\int_0^{T_{\epsilon_1}} u^2\,dt$
& $T_{\epsilon_2}$ & $\sigma_{\text{nom}}$ & $\int_0^{T_{\epsilon_2}} u^2\,dt$
& $T_{\epsilon_3}$ & $\sigma_{\text{nom}}$ & $\int_0^{T_{\epsilon_3}} u^2\,dt$
\\
\midrule
Linear   & 1.535 & 3.000 & 7.833 & 1.998 & 3.000 & 7.864 & 3.076 & 3.000 & 7.875 & 9.938 \\
$r=1.0$  & 0.860 & 5.361 & 7.501 & 1.196 & 5.779 & 7.502 & 1.535 & 6.004 & 7.502 & 9.938 \\
$r=0.9$  & 0.899 & 5.127 & 7.337 & 1.235 & 5.594 & 7.338 & 1.574 & 5.853 & 7.338 & 9.317 \\
$r=0.8$  & 0.948 & 4.858 & 7.265 & 1.286 & 5.377 & 7.266 & 1.624 & 5.673 & 7.266 & 8.697 \\
$r=0.7$  & 1.013 & 4.547 & 7.313 & 1.351 & 5.115 & 7.314 & 1.690 & 5.452 & 7.314 & 8.078 \\
$r=0.6$  & 1.102 & 4.181 & 7.530 & 1.441 & 4.797 & 7.531 & 1.780 & 5.178 & 7.531 & 7.455 \\
\bottomrule
\end{tabular}
\end{table*}

\subsection{Quadrotor attitude under actuator limits}\label{subsec:quad}
We consider the quadrotor model of \cite{lee2010geometric,lee2010control}:
\begin{align}
\dot{x}_p&=x_v,\\
m\dot{x}_v&=m g e_3 - f R e_3,\\
\dot{R} &= R\hat{\omega},\\
J\dot{\omega}&=-\omega\times J\omega + u,
\end{align}
where $x_p,x_v\in\mathbb{R}^3$, $m\in\mathbb{R}$, $\omega\in\mathbb{R}^3$, $f\in\mathbb{R}$, $R\in \mathfrak{so}(3)$, $e_3=[0,0,1]^\top$, $J\in\mathbb{R}^{3\times3}$, $u\in\mathbb{R}^3$ with $\|u\|_{\infty}\le\theta$, and
\[
    \hat{\omega} = \begin{bmatrix}
        0 & -\omega_3 &\omega_2\\
        \omega_3 & 0 &-\omega_1\\
        -\omega_2 & \omega_1 & 0
    \end{bmatrix}.
\]
We compare the concave CLF–QP with the flexible CLF–QP presented in \cite{FlexCLF}, which is effective for attitude control. We adopt the same CLF from \cite{FlexCLF}:
\[
    V(x)= \frac{1}{2}e_{\omega}^\top J e_{\omega} + k_{R}\,\Psi(R,R_d) + k_c\,e_{R}^\top e_{\omega},
\]
where $e_{\omega}=\omega-R^\top R_d\omega_d$, $\Psi(R,R_d)=\tfrac{1}{2}\mathrm{tr}[I-R_d^\top R]$,
and $e_{R}=\tfrac{1}{2}(R_d^\top R-R^\top R_d)^\vee$ with the \emph{vee} map
$\vee:\mathfrak{so}(3)\to \mathbb{R}^3$   (see \cite{lee2010geometric,lee2010control}).

\subsubsection*{Parameter settings.}
We focus on attitude control with
\[
  J=\mathrm{diag}(0.0820,0.0845,0.1377)\ \mathrm{kg\,m^2}, \quad m=4.34\ \mathrm{kg},
\]
\[
  R(0) = \begin{bmatrix}
      0.2500 &-0.0580 & 0.9665\\
      0.4330 & 0.8995 &-0.0580\\
      -0.8660 & 0.4330 &0.2500
  \end{bmatrix},\quad \omega(0)=0,
\]
\[
  R_d = I,\ \omega_d = 0,\ k_R= 8.81,\ k_c=0.1377.
\]
All parameters match \cite{FlexCLF}. For flexible CLF–QP, $\kappa(x)= 0.9\left(1-e^{-0.9V(x)}\right)$ and $[\sigma_{\min},\sigma_{\max}]=[0.29,\,10]$ from \cite{FlexCLF}.
For the linear or concave CLF–QP, $H=J^{-1}$ and $q=300$. The baseline $\sigma=2$ is chosen to match \cite{FlexCLF}. 
We use the normalized rational concave factor with $k_{\min}=0.8$, $k_{\max}=2.5$, $r=\{0.85,0.95\}$ (so $s(c)=r$).
Sampling $\mathrm dt=1$\,ms. Simulations use the Drake Python library. 
To compare actuation demand, we set $\theta=11$ to qualify the baseline controller and report the realized $\|u\|_\infty$.

\subsubsection*{Results}
Table~\ref{tab:quad_metrics} summarizes the windowed nominal rates and energy. Concave CLF–QP gains higher $\sigma_{\mathrm{nom}}$ at comparable or lower energy and \emph{lower peak torque} as shown in Fig.~\ref{fig:quad_normV}.

\begin{table}[t]
\centering
\caption{Quadrotor Attitude: Performance Metrics}
\label{tab:quad_metrics}
\begin{tabular}{l cc cc c}
\toprule
& \multicolumn{2}{c}{$\epsilon_1 = 10^{-2}c$}
& \multicolumn{2}{c}{$\epsilon_2 = 10^{-3}c$}
& $\|u\|_\infty$ \\
\cmidrule(lr){2-3} \cmidrule(lr){4-5}
Controller 
& $\sigma_{\text{nom}}$ & $\int_0^{T_{\epsilon_1}} u^2\, dt$
& $\sigma_{\text{nom}}$ & $\int_0^{T_{\epsilon_2}} u^2\, dt$
& -- \\
\midrule
Flexible    & 2.835 & 0.921 & 2.827 & 0.922 & 10.899 \\
$r=0.95$    & 4.523 & 1.049 & 3.358 & 1.050 & ~7.899 \\
$r=0.85$    & 4.344 & 0.785 & 3.250 & 0.787 & ~7.068 \\
\bottomrule
\end{tabular}
\end{table}
\begin{figure}
    \centering
    \includegraphics[width=1\linewidth]{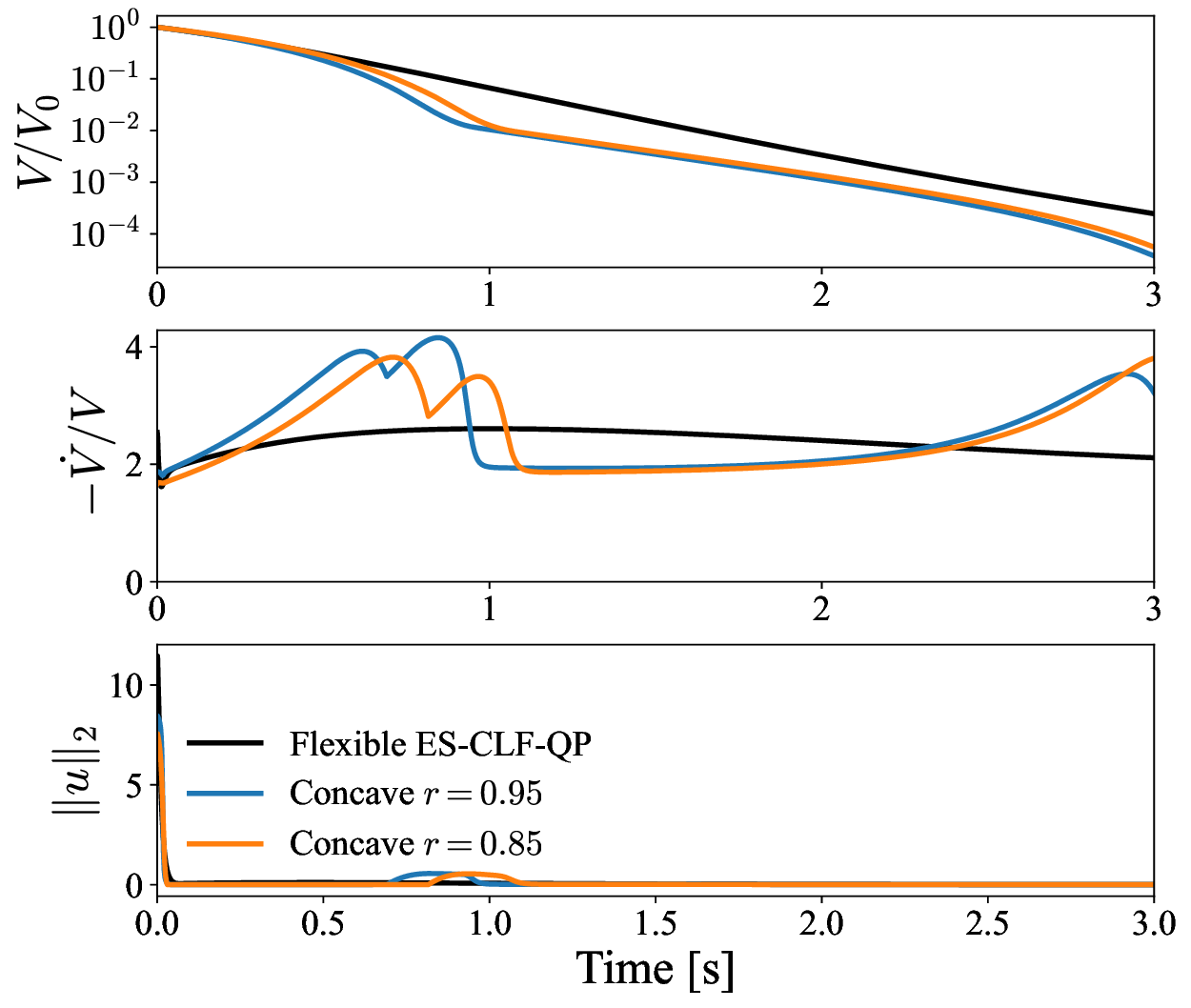}
    \caption{Quadrotor: normalized Lyapunov response $V(x(t))/V(x_0)$, instantaneous exponential rate $-\dot V/V$, and control input norm $\|u\|_2$. The concave CLF–QP yields faster decay than the flexible ES–CLF–QP after $t\approx0.5$\,s; along most of the trajectory, the instantaneous rate exceeds the baseline $\sigma=2$, consistent with the feasibility-preserving acceleration theorem. With $r=0.85$, faster decay is achieved with a lower peak input $\|u\|_\infty$ and reduced energy. Decreasing $r$ further reduces the actuation peak, consistent with endpoint-dominated actuation level.}
    \label{fig:quad_normV}
\end{figure}

\section{Conclusion}
We showed that shaping $\alpha$ rather than only scaling it is crucial for CLF design under actuator limits. Using the windowed nominal rate $\sigma_{\alpha}(\epsilon,c)$ and the endpoint cap $\bar\alpha(\theta,c)$, strictly \emph{concave} comparisons deliver feasibility-preserving acceleration. A rational, Lipschitz-aware factor yields a practical CLF–QP design with closed-form tuning, and case studies confirm faster guaranteed decay with reduced actuation level and energy. Beyond CLF–QP, the approach \emph{applies to any Lyapunov-based controller} that enforces a comparison inequality (e.g., backstepping, passivity, CCM, LMI). As evidence, our prior passivity-based controller for a 12-D saturated underwater vehicle benefits from concave shaping, improving tracking and robustness \cite{fan2024control}. Future work includes CLF/CBF co-design, and sampling-aware and slack-aware slope limits.

\section*{References}
\bibliographystyle{IEEEtran}
\bibliography{ref}

@inproceedings{ames2013towards,
  title={Towards the unification of locomotion and manipulation through control lyapunov functions and quadratic programs},
  author={Ames, Aaron D and Powell, Matthew},
  booktitle={Control of Cyber-Physical Systems: Workshop held at Johns Hopkins University, March 2013},
  pages={219--240},
  year={2013},
  organization={Springer}
}

@article{mestres2025regularity,
  title={Regularity properties of optimization-based controllers},
  author={Mestres, Pol and Allibhoy, Ahmed and Cort{\'e}s, Jorge},
  journal={European Journal of Control},
  volume={81},
  pages={101098},
  year={2025},
  publisher={Elsevier}
}

@ARTICLE{numeicalchattering,
  author={Acary, Vincent and Brogliato, Bernard and Orlov, Yury V.},
  journal={IEEE Transactions on Automatic Control}, 
  title={Chattering-Free Digital Sliding-Mode Control With State Observer and Disturbance Rejection}, 
  year={2012},
  volume={57},
  number={5},
  pages={1087-1101},
  doi={10.1109/TAC.2011.2174676}}

@inproceedings{lee2010geometric,
  title={Geometric tracking control of a quadrotor UAV on SE (3)},
  author={Lee, Taeyoung and Leok, Melvin and McClamroch, N Harris},
  booktitle={49th IEEE conference on decision and control (CDC)},
  pages={5420--5425},
  year={2010},
  organization={IEEE}
}

@article{lee2010control,
  title={Control of complex maneuvers for a quadrotor UAV using geometric methods on SE (3)},
  author={Lee, Taeyoung and Leok, Melvin and McClamroch, N Harris},
  journal={arXiv preprint arXiv:1003.2005},
  year={2010}
}

@article{ames2014rapidly,
  title={Rapidly exponentially stabilizing control lyapunov functions and hybrid zero dynamics},
  author={Ames, Aaron D and Galloway, Kevin and Sreenath, Koushil and Grizzle, Jessy W},
  journal={IEEE Transactions on Automatic Control},
  volume={59},
  number={4},
  pages={876--891},
  year={2014},
  publisher={IEEE}
}

@ARTICLE{Fixtime,
  author={Polyakov, Andrey},
  journal={IEEE Transactions on Automatic Control}, 
  title={Nonlinear Feedback Design for Fixed-Time Stabilization of Linear Control Systems}, 
  year={2012},
  volume={57},
  number={8},
  pages={2106-2110},
  doi={10.1109/TAC.2011.2179869}}

@ARTICLE{FTC,
  author={Bhat, S.P. and Bernstein, D.S.},
  journal={IEEE Transactions on Automatic Control}, 
  title={Continuous finite-time stabilization of the translational and rotational double integrators}, 
  year={1998},
  volume={43},
  number={5},
  pages={678-682},
  doi={10.1109/9.668834}}

@misc{FlexCLF, author = {Logan Dihel and Nak-Seung P. Hyun}, title = {Flexible Convergence Rate for Quadratic Programming-Based Control Lyapunov Functions}, journal = {2024 IEEE 63rd Conference on Decision and Control (CDC)}, year = {2024}, doi = {10.1109/CDC56724.2024.10886305}, pages = {8845-8851} }

@ARTICLE{LMI,
  author={Chilali, M. and Gahinet, P.},
  journal={IEEE Transactions on Automatic Control}, 
  title={H/sub /spl infin// design with pole placement constraints: an LMI approach}, 
  year={1996},
  volume={41},
  number={3},
  pages={358-367},
  doi={10.1109/9.486637}}

@ARTICLE{CCM,
  author={Manchester, Ian R. and Slotine, Jean-Jacques E.},
  journal={IEEE Transactions on Automatic Control}, 
  title={Control Contraction Metrics: Convex and Intrinsic Criteria for Nonlinear Feedback Design}, 
  year={2017},
  volume={62},
  number={6},
  pages={3046-3053},
  doi={10.1109/TAC.2017.2668380}}

@INPROCEEDINGS{inverseCLF,
  author={Reher, Jenna and Kann, Claudia and Ames, Aaron D.},
  booktitle={2020 American Control Conference (ACC)}, 
  title={An Inverse Dynamics Approach to Control Lyapunov Functions}, 
  year={2020},
  volume={},
  number={},
  pages={2444-2451},
  doi={10.23919/ACC45564.2020.9147342}}

@inproceedings{ames2019control,
  title={Control barrier functions: Theory and applications},
  author={Ames, Aaron D and Coogan, Samuel and Egerstedt, Magnus and Notomista, Gennaro and Sreenath, Koushil and Tabuada, Paulo},
  booktitle={2019 18th European control conference (ECC)},
  pages={3420--3431},
  year={2019},
  organization={Ieee}
}

@ARTICLE{CLF-bipedal,
  author={Galloway, Kevin and Sreenath, Koushil and Ames, Aaron D. and Grizzle, Jessy W.},
  journal={IEEE Access}, 
  title={Torque Saturation in Bipedal Robotic Walking Through Control Lyapunov Function-Based Quadratic Programs}, 
  year={2015},
  volume={3},
  number={},
  pages={323-332},
  doi={10.1109/ACCESS.2015.2419630}}

@ARTICLE{CLFreview,
  author={Li, Boqian and Wen, Shiping and Yan, Zheng and Wen, Guanghui and Huang, Tingwen},
  journal={IEEE/CAA Journal of Automatica Sinica}, 
  title={A Survey on the Control Lyapunov Function and Control Barrier Function for Nonlinear-Affine Control Systems}, 
  year={2023},
  volume={10},
  number={3},
  pages={584-602},
  doi={10.1109/JAS.2023.123075}}

@INPROCEEDINGS{OD-CLF-QP,
  author={Shahraki, Milad Alipour and Lessard, Laurent},
  booktitle={2025 American Control Conference (ACC)}, 
  title={Spacecraft Attitude Control Under Reaction Wheel Constraints Using Control Lyapunov and Control Barrier Functions}, 
  year={2025},
  volume={},
  number={},
  pages={947-952},
  doi={10.23919/ACC63710.2025.11108025}}

@article{sontag1989universal,
  title={A ‘universal’construction of Artstein's theorem on nonlinear stabilization},
  author={Sontag, Eduardo D},
  journal={Systems \& control letters},
  volume={13},
  number={2},
  pages={117--123},
  year={1989},
  publisher={Elsevier}
}

@article{ArtsteinCLF,
title = "Stabilization with relaxed controls",
author = "Zvi Artstein",
year = "1983",
doi = "10.1016/0362-546X(83)90049-4",
language = "English",
volume = "7",
pages = "1163--1173",
journal = "Nonlinear Analysis-Theory Methods \& Applications",
issn = "0362-546X",
publisher = "Elsevier",
number = "11",
}

@misc{choi2020cbfclfhelper,
  author       = {Jason J. Choi},
  title        = {CBF-CLF-Helper 1.0: Library for Control Barrier Function (CBF) and Control Lyapunov Function (CLF) based control methods},
  year         = {2020},
  version      = {1.0.0},
  url          = {https://github.com/HybridRobotics/CBF-CLF-Helper}
}

@INPROCEEDINGS{constraintCLF,
  author={Mahmood, Maaz and Mhaskar, Prashant},
  booktitle={Proceedings of the 2010 American Control Conference}, 
  title={On constructing constrained control Lyapunov functions for linear systems}, 
  year={2010},
  volume={},
  number={},
  pages={5191-5196},
  doi={10.1109/ACC.2010.5530693}}

@article{Mhaskar2005,
  author    = {Pramod Mhaskar and Panagiotis D. Christofides},
  title     = {Stabilization of nonlinear systems with state and control constraints},
  journal   = {Automatica},
  volume    = {41},
  number    = {9},
  pages     = {1521--1531},
  year      = {2005},
  doi       = {10.1016/j.automatica.2005.04.001}
}

@article{Mhaskar2006,
  author   = {Pramod Mhaskar and Panagiotis D. Christofides},
  title     = {Robust control of nonlinear systems with bounded uncertainties},
  journal   = {Systems \& Control Letters},
  volume    = {55},
  number    = {1},
  pages     = {52--61},
  year      = {2006},
  doi       = {10.1016/j.sysconle.2005.04.004}
}

@article{lin1996smooth,
  title={A smooth converse Lyapunov theorem for robust stability},
  author={Lin, Yuandan and Sontag, Eduardo D and Wang, Yuan},
  journal={SIAM Journal on Control and Optimization},
  volume={34},
  number={1},
  pages={124--160},
  year={1996},
  publisher={SIAM}
}

@incollection{Mironchenkbook,
  title={Input-to-state stability},
  author={Mironchenko, Andrii},
  booktitle={Input-to-State Stability: Theory and Applications},
  pages={41--115},
  year={2023},
  publisher={Springer}
}

@inproceedings{fan2024control,
  title={Control of Underactuated Autonomous Underwater Vehicles With Input Saturation Based on Passivity Using Feedback Concavification},
  author={Fan, Shuyuan and Werner, Herbert},
  booktitle={2024 European Control Conference (ECC)},
  pages={3233--3238},
  year={2024},
  organization={IEEE}
}

@ARTICLE{ChatteringAnalysis,
  author={Levant, Arie},
  journal={IEEE Transactions on Automatic Control}, 
  title={Chattering Analysis}, 
  year={2010},
  volume={55},
  number={6},
  pages={1380-1389},
  doi={10.1109/TAC.2010.2041973}}

@ARTICLE{LevantTACSMC,
  author={Levant, A.},
  journal={IEEE Transactions on Automatic Control}, 
  title={Quasi-continuous high-order sliding-mode controllers}, 
  year={2005},
  volume={50},
  number={11},
  pages={1812-1816},
  doi={10.1109/TAC.2005.858646}}

@book{khalil2002nonlinear,
  title={Nonlinear systems},
  author={Khalil, Hassan K and Grizzle, Jessy W},
  volume={3},
  year={2002},
  publisher={Prentice hall Upper Saddle River, NJ}
}

\end{document}